\documentclass[12pt]{article}
\usepackage[T1]{fontenc}
\usepackage{amsmath,amssymb,amsthm}
\usepackage{tikz}
\usepackage[letterpaper,margin=1in]{geometry}
\usepackage{booktabs,microtype,xurl}
\usepackage[hidelinks]{hyperref}
\hypersetup{pdftitle={Rational Dyck Excursion Counts as Sums of Generalized Fuss-Catalan Products},pdfauthor={Brecht Verbeken, Arne Vanhoyweghen, Vincent Ginis}}
\allowdisplaybreaks[2]
\newtheorem{theorem}{Theorem}
\newtheorem{proposition}[theorem]{Proposition}
\newtheorem{corollary}[theorem]{Corollary}
\newtheorem{lemma}[theorem]{Lemma}
\newcommand{\Q}{\mathbb{Q}}
\newcommand{\Z}{\mathbb{Z}}

\newcommand{\seqnum}[1]{\href{https://oeis.org/#1}{\textup{\underline{#1}}}}

\title{Rational Dyck Excursion Counts as Sums of Generalized Fuss--Catalan Products}
\author{%
  Brecht Verbeken\textsuperscript{1,2},
  Arne Vanhoyweghen\textsuperscript{1,2}, and
  Vincent Ginis\textsuperscript{1,2,3}\\[0.8em]
  \parbox{0.96\textwidth}{\centering\small
    \textsuperscript{1}Department of Business Technology and Operations,\\
    Data Analytics Laboratory, Vrije Universiteit Brussel (VUB),\\
    Pleinlaan 2, 1050 Brussels, Belgium.\\[0.45em]
    \textsuperscript{2}imec-SMIT, Vrije Universiteit Brussel,\\
    Pleinlaan 9, 1050 Brussels, Belgium.\\[0.45em]
    \textsuperscript{3}School of Engineering and Applied Sciences,\\
    Harvard University, Cambridge, Massachusetts 02138, USA.
  }%
}
\date{}
\begin{document}
\maketitle
\begin{abstract}
For walks with jumps $+p$ and $-q$, where $p,q$ are coprime positive integers, we express the weak and strict rational Dyck excursion counts as finite weighted sums of generalized Fuss--Catalan products. Both sums run over the same partitions of $q$ for every path size. Reversing the order of the steps and changing their signs transforms a returning walk with jumps $+p$ and $-q$ into one with jumps $+q$ and $-p$, with the same heights in reverse order. This preserves both weak and strict excursions. Applying the formulas with $p$ and $q$ exchanged therefore gives sums over partitions of $p$, so we can choose the smaller of the two jump magnitudes. The indices of the Fuss--Catalan factors satisfy certain divisibility conditions that are determined by the jump magnitudes and the partition parts. The weak formula alternates in sign with the number of parts, whereas every term in the strict formula is positive. We prove both formulas by cutting unrestricted walks at their first and last visits to their lowest height. Expanding the resulting factors produces the partitions of $q$ and the different signs. We also give equivalent binomial convolutions. Multiplying each Fuss--Catalan factor by $q$ in the weak formula or $p$ in the strict formula gives a positive integer. For $q=1$, the weak and strict counts reduce to Fuss--Catalan and shifted Raney numbers, respectively; for $q=2$ and $q=3$, the formulas yield explicit convolution identities.
\end{abstract}

\section{Introduction}\label{secintroduction}
A nonempty walk can return to zero only if its increments have opposite signs. We write these increments as $+p$ and $-q$, where $p$ and $q$ are coprime positive integers denoting the jump magnitudes. A weak excursion starts and ends at zero and stays nonnegative; a strict excursion is positive between its endpoints. A weak meander starts at zero and stays nonnegative, but need not return to zero. A positive meander starts at zero and stays strictly positive at every later step. In the special case where the downward jump has magnitude one ($q=1$), the weak and strict excursion counts are given by Fuss--Catalan and shifted Raney numbers, respectively. We seek expressions for these counts in terms of the same coefficient family when $q>1$. We give weighted sums of generalized Fuss--Catalan products. These formulas sum over partitions of the jump magnitude $q$, rather than partitions of the path size. For each partition, the indices in the products must satisfy certain divisibility conditions.

For original jump magnitudes $p_0,q_0$, put $g=\gcd(p_0,q_0)$ and $p=p_0/g$, $q=q_0/g$. We call this \emph{normalization}; dividing all heights by $g$ preserves the weak and strict excursion counts. After normalization, a returning walk of size $n$ has $qn$ upward steps and $pn$ downward steps, hence $(p+q)n$ steps altogether.

Figure~\ref{figwalkexamples} shows three walks with jumps $+4$ and $-3$ that return to zero after fourteen steps. The first is strict, the second is weak with an interior return to zero, and the third goes below zero.

\begin{figure}[t]
\centering
\definecolor{walkstrict}{HTML}{156F79}
\definecolor{walkweak}{HTML}{B97514}
\definecolor{walkbridge}{HTML}{B54F4B}
\definecolor{walkgrid}{HTML}{DDE2E6}
\begin{minipage}[t]{0.32\linewidth}
\centering
{\small\bfseries (a) Strict excursion}\par\smallskip
\begin{tikzpicture}[x=0.29cm,y=0.29cm,font=\scriptsize]
\path[use as bounding box] (-1.5,-5.15) rectangle (14.7,9.5);
\fill[black!2] (0,-3) rectangle (14,9);
\foreach \y in {-3,3,6,9} {\draw[walkgrid,line width=0.25pt] (0,\y)--(14,\y);}
\draw[black!60,line width=0.65pt] (0,0)--(14,0);
\draw[black!45,line width=0.35pt] (0,-3)--(0,9);
\foreach \y in {-3,0,3,6,9} {\draw[black!50] (-0.13,\y)--(0,\y);\node[left,inner sep=2pt,text=black!70] at (-0.13,\y) {$\y$};}
\foreach \x in {0,7,14} {\node[below,text=black!70] at (\x,-3.1) {$\x$};}
\node[text=black!70] at (7,-4.85) {Step};
\draw[walkstrict,line width=1.2pt,line join=round] plot coordinates {(0,0) (1,4) (2,8) (3,5) (4,2) (5,6) (6,3) (7,7) (8,4) (9,8) (10,5) (11,2) (12,6) (13,3) (14,0)};
\foreach \p in {(0,0),(1,4),(2,8),(3,5),(4,2),(5,6),(6,3),(7,7),(8,4),(9,8),(10,5),(11,2),(12,6),(13,3),(14,0)} {\fill[walkstrict] \p circle[radius=1.05pt];}
\fill[walkstrict] (0,0) circle[radius=1.65pt];
\fill[walkstrict] (14,0) circle[radius=1.65pt];
\end{tikzpicture}\par
{\scriptsize Positive between its endpoints}
\end{minipage}\hfill
\begin{minipage}[t]{0.32\linewidth}
\centering
{\small\bfseries (b) Weak, not strict}\par\smallskip
\begin{tikzpicture}[x=0.29cm,y=0.29cm,font=\scriptsize]
\path[use as bounding box] (-1.5,-5.15) rectangle (14.7,9.5);
\fill[black!2] (0,-3) rectangle (14,9);
\foreach \y in {-3,3,6,9} {\draw[walkgrid,line width=0.25pt] (0,\y)--(14,\y);}
\draw[black!60,line width=0.65pt] (0,0)--(14,0);
\draw[black!45,line width=0.35pt] (0,-3)--(0,9);
\foreach \y in {-3,0,3,6,9} {\draw[black!50] (-0.13,\y)--(0,\y);\node[left,inner sep=2pt,text=black!70] at (-0.13,\y) {$\y$};}
\foreach \x in {0,7,14} {\node[below,text=black!70] at (\x,-3.1) {$\x$};}
\node[text=black!70] at (7,-4.85) {Step};
\draw[walkweak,line width=1.2pt,line join=round] plot coordinates {(0,0) (1,4) (2,8) (3,5) (4,2) (5,6) (6,3) (7,0) (8,4) (9,8) (10,5) (11,2) (12,6) (13,3) (14,0)};
\foreach \p in {(0,0),(1,4),(2,8),(3,5),(4,2),(5,6),(6,3),(7,0),(8,4),(9,8),(10,5),(11,2),(12,6),(13,3),(14,0)} {\fill[walkweak] \p circle[radius=1.05pt];}
\fill[walkweak] (0,0) circle[radius=1.65pt];
\fill[walkweak] (14,0) circle[radius=1.65pt];
\draw[walkweak,line width=1pt,fill=white] (7,0) circle[radius=2.6pt];
\end{tikzpicture}\par
{\scriptsize An interior return at step 7}
\end{minipage}\hfill
\begin{minipage}[t]{0.32\linewidth}
\centering
{\small\bfseries (c) Returning walk}\par\smallskip
\begin{tikzpicture}[x=0.29cm,y=0.29cm,font=\scriptsize]
\path[use as bounding box] (-1.5,-5.15) rectangle (14.7,9.5);
\fill[black!2] (0,-3) rectangle (14,9);
\fill[walkbridge!7] (0,-3) rectangle (14,0);
\foreach \y in {-3,3,6,9} {\draw[walkgrid,line width=0.25pt] (0,\y)--(14,\y);}
\draw[black!60,line width=0.65pt] (0,0)--(14,0);
\draw[black!45,line width=0.35pt] (0,-3)--(0,9);
\foreach \y in {-3,0,3,6,9} {\draw[black!50] (-0.13,\y)--(0,\y);\node[left,inner sep=2pt,text=black!70] at (-0.13,\y) {$\y$};}
\foreach \x in {0,7,14} {\node[below,text=black!70] at (\x,-3.1) {$\x$};}
\node[text=black!70] at (7,-4.85) {Step};
\draw[walkbridge,line width=1.2pt,line join=round] plot coordinates {(0,0) (1,4) (2,1) (3,-2) (4,2) (5,6) (6,3) (7,0) (8,4) (9,8) (10,5) (11,2) (12,6) (13,3) (14,0)};
\foreach \p in {(0,0),(1,4),(2,1),(3,-2),(4,2),(5,6),(6,3),(7,0),(8,4),(9,8),(10,5),(11,2),(12,6),(13,3),(14,0)} {\fill[walkbridge] \p circle[radius=1.05pt];}
\fill[walkbridge] (0,0) circle[radius=1.65pt];
\fill[walkbridge] (14,0) circle[radius=1.65pt];
\draw[walkbridge,line width=1pt,fill=white] (3,-2) circle[radius=2.6pt];
\end{tikzpicture}\par
{\scriptsize Falls below zero at step 3}
\end{minipage}
\caption{Walks with jumps $+4$ and $-3$. Each has six upward and eight downward steps and returns to zero after fourteen steps. (a) A strict excursion. (b) A weak excursion with an interior return to zero at step 7. (c) A returning walk that reaches height $-2$ and is therefore not an excursion. The vertical coordinate records height; the horizontal coordinate records the number of steps.}
\label{figwalkexamples}
\end{figure}
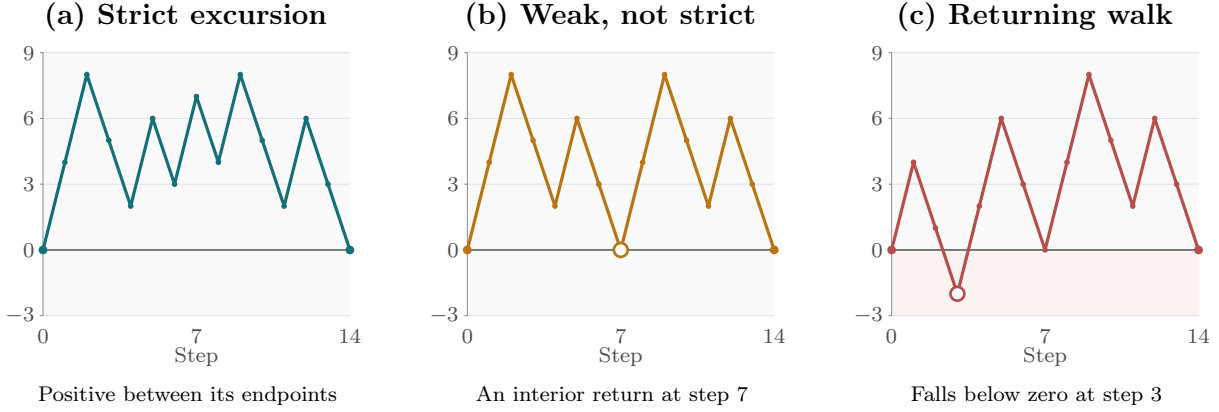

These formulas include Duchon's sequence \seqnum{A060941}, the corresponding strict excursion sequence \seqnum{A293946} at positive indices, and the weak excursion sequences \seqnum{A300386}, \seqnum{A300390}, and \seqnum{A300391} in the On-Line Encyclopedia of Integer Sequences~\cite{OEIS}. When one jump has magnitude one, they recover the integer-parameter Fuss--Catalan and Raney sequences. The contribution is a common product representation for these established excursion sequences.

To express the excursion counts in one coefficient family, we use the two-parameter Fuss--Catalan coefficients. For $\alpha>1$, let $F_\alpha(x)$ be the formal power series with constant term one satisfying
\[
 F_\alpha=1+xF_\alpha^\alpha.
\]
For $\beta>0$ and a nonnegative integer $m$, define
\begin{equation}\label{eqfcintro}
 C_{\alpha,\beta}(m)=[x^m]F_\alpha(x)^\beta
 =\frac{\beta}{\alpha m+\beta}\binom{\alpha m+\beta}{m}.
\end{equation}
Brackets denote coefficient extraction: $[x^m]f(x)$ is the coefficient of $x^m$ in $f(x)$. With two variables, $[t^a x^b]f(t,x)$ is the coefficient of $t^a x^b$ in $f(t,x)$. The same convention applies with any choice of variables. When $\alpha$ is an integer, $\beta=1$ gives the usual Fuss--Catalan numbers, and positive integer $\beta$ gives their Raney extension. Our formulas use positive rational values of $\alpha$ and $\beta$, so the individual coefficients need not be integers.

Raney~\cite{Raney} develops the composition and reversion framework behind these coefficients, and Gessel~\cite[Section~3.3]{Gessel2016} gives the power identity in~\eqref{eqfcintro}. Lin~\cite{Lin} treats their integer-parameter interpretations. Partition-indexed Fuss--Catalan product identities also arise from the determinant expansions of Goy and Shattuck~\cite{GoyShattuck}. Here the partitioned integer is the jump $q$, and the parameters and residue restrictions of the factors are determined by $p,q$ and the partition parts.

The enumeration of rational Dyck paths is well known in the literature. Bizley~\cite{Bizley} proves Grossman's formula and counts paths with prescribed numbers of boundary contacts, covering both weak and strict excursions. Birmajer, Gil, and Weiner~\cite{Birmajer} express these counts using Bell polynomials. The factor-free words studied in these works and by Duchon~\cite{Duchon} satisfy a stronger restriction than strict excursions. A strict excursion starts and ends at zero and stays strictly positive in between. Within such an excursion, a shorter segment may start and end at the same positive height and stay at or above that height throughout. Such a segment is a Dyck factor above zero.

Put $h=p+q$, and let $E(z)$ and $T(z)$ be the generating functions for weak and nonempty strict excursions, respectively, with $z^n$ marking length $hn$. The empty path contributes one to $E(z)$. Cutting a weak excursion at its interior returns to zero gives a unique ordered sequence of strict excursions. Multiplication counts successive choices of pieces, so
\[
 E(z)=1+T(z)+T(z)^2+\cdots=\frac{1}{1-T(z)}.
\]
The Grossman--Bizley formula for $E(z)$ and the resulting expression for $T(z)$ are
\begin{equation}\label{eqbizley}
 E(z)=\exp\!\left(\sum_{j\ge1}\frac{1}{hj}\binom{hj}{qj}z^j\right),
 \qquad T(z)=1-E(z)^{-1}.
\end{equation}
Expanding the exponential results in sums indexed by partitions of the path size $n$. Theorem~\ref{thmmain} instead indexes the products by partitions of the jump magnitude $q$: the weak formula alternates in sign with the number of partition parts, whereas every strict summand is positive.

The proof uses the walk factorization of Gessel~\cite{Gessel1980} and Xin~\cite{Xin}, obtained by cutting an unrestricted walk at its first and last visits to its minimum height. The middle piece is a weak excursion relative to that height. Deleting the final $-q$ step of a strict excursion leaves a positive walk ending at $q$. Expanding the positive endpoint factor distributes the endpoint height $q$ among partition parts and gives the strict product formula. The weak formula comes from the reciprocal of the negative endpoint factor, whose expansion produces the alternating signs. Section~\ref{secstructure} gives a combinatorial interpretation of the integer multiples of the Fuss--Catalan factors.

Related convolution formulas have been obtained using the kernel method for paths below a line of rational slope. Sato's matrix formulation~\cite{Sato} is reproduced by Krattenthaler~\cite[Section~10.5]{Krattenthaler}; Banderier and Flajolet~\cite[Example~5]{BF} give a convolution of generalized binomial coefficients whose number of factors is independent of the path size, and attribute it to Sato. Banderier and Wallner~\cite{BW} use roots of unity and Lagrange inversion, while Bousquet-M\'elou~\cite{Bousquet} expresses elementary symmetric functions in terms of power sums. In contrast, here we express weak and strict excursion counts as Fuss--Catalan products with explicit congruence conditions and positive parameters in the strict formula.

Panfil, Sto\v{s}i\'c, and Su\l{}kowski~\cite{Panfil} reproduce a root-of-unity convolution in their study of lattice paths and torus knots. Ek~\cite{Ek} develops computational methods for lattice-walk generating functions, and Firoozi, Jedwab, and Rattan~\cite{Firoozi} count paths according to the number of lattice points above a boundary of rational slope. The minimal bridges studied by Lou and Brown~\cite{LouBrown} avoid interior lattice points on the boundary but may cross it, whereas strict excursions stay strictly on one side between their endpoints.

The authors' original motivation came from additive wealth processes~\cite{VG,VVMG}, which also provide an application context. For independent increments $+p$ and $-q$ with fixed probabilities, all walks with fixed numbers of upward and downward steps have the same probability of occurring. In Section~\ref{secprobability}, we use this fact to convert the excursion counts into probabilities of weak and strict returns. We also give probabilities for first entry at or below zero from initial height $p$, distinguishing exact landings at zero from overshoots.

Section~\ref{secmain} states the product formulas; Section~\ref{secproof} derives them from endpoint factorization and then gives their shifted forms. Their integer scalings are developed in Section~\ref{secstructure}. Sections~\ref{secexamples} and~\ref{secverification} give sequence specializations and OEIS indexing. Appendix~\ref{appformulas} collects expanded binomial formulas.

\section{Fuss--Catalan product formulas}\label{secmain}
Throughout, $p,q$ are coprime positive integers and $h=p+q$. We index returning walks by their size $n$, so size $n$ means $hn$ steps. For a walk represented by an increment sequence $(x_1,\ldots,x_{hn})$ in $\{p,-q\}$, put $H_0=0$ and $H_j=x_1+\cdots+x_j$ for $1\le j\le hn$. For $n\ge0$, let $E_n^{(p,q)}$ count weak excursions of size $n$, that is, walks with $H_{hn}=0$ and $H_j\ge0$ for $0\le j\le hn$. For $n\ge1$, let $T_n^{(p,q)}$ count strict excursions of size $n$, with $H_{hn}=0$ and $H_j>0$ for $1\le j<hn$. We take $E_0^{(p,q)}=1$ and $T_0^{(p,q)}=0$. Let $u$ and $d$ be the numbers of upward and downward steps. A return to zero requires $pu=qd$, and the total length gives $u+d=hn$. Together these equations yield $u=qn$ and $d=pn$. Replacing each upward step by an east step and each downward step by a north step gives a lattice path from $(0,0)$ to $(qn,pn)$. At a vertex $(x,y)$, the height of the original walk is $px-qy$. This implies that a weak excursion corresponds to a path on or below the line $y=(p/q)x$. Furthermore, a strict excursion corresponds to a path that starts and ends on this line and stays strictly below it in between.

Write $\lambda=(\lambda_1,\ldots,\lambda_\ell)\vdash q$ for a partition of $q$ into $\ell=\ell(\lambda)$ positive parts, listed in nonincreasing order. Let $m_r(\lambda)$ be the number of occurrences of the part $r$, and define
\begin{equation}\label{eqzlambda}
 m_r(\lambda)=\#\{j:\lambda_j=r\},
 \qquad z_\lambda=\prod_{r=1}^q r^{m_r(\lambda)}m_r(\lambda)!.
\end{equation}
\par\noindent\begin{minipage}{\linewidth}
For $q=3$, we obtain:
\begin{center}
\begin{tabular}{@{}cccc@{}}
\toprule
$\lambda$ & $\ell(\lambda)$ & Nonzero multiplicities & $z_\lambda$\\
\midrule
$(3)$ & $1$ & $m_3=1$ & $3^1\,1!=3$\\
$(2,1)$ & $2$ & $m_2=m_1=1$ & $2^1\,1!\,1^1\,1!=2$\\
$(1,1,1)$ & $3$ & $m_1=3$ & $1^3\,3!=6$\\
\bottomrule
\end{tabular}
\par\smallskip
{\small Partitions of $3$, their lengths and nonzero multiplicities, and the corresponding values of $z_\lambda$.\par}
\end{center}
\end{minipage}\par
In the last row, $m_1=3$ gives the factor $3!$; all other multiplicities are zero. For every path size $n$, the formulas below sum over these same partitions of $q$. All summation indices are nonnegative integers. A sum over an empty index set is zero, and each displayed congruence is imposed separately on every factor. Congruences modulo~$1$ impose no restriction.

We use the coefficients $C_{\alpha,\beta}(m)$ defined in~\eqref{eqfcintro}. In the weak formula, their parameters are $\alpha=h/q$ and $\beta=\lambda_j/q$, and their index $u_j$ is the number of upward steps in the associated auxiliary walk. In the strict formula, their parameters are $\alpha=h/p$ and $\beta=\lambda_j/p$, and their index $d_j$ is the number of downward steps in its auxiliary walk.

\begin{theorem}[Fuss--Catalan product formulas]\label{thmmain}
For $n\ge0$, the weak excursion number is
\begin{equation}\label{eqmainweak}
\displaystyle
 E_n^{(p,q)}=
 \sum_{\lambda\vdash q}\frac{(-1)^{\ell-1}q^\ell}{z_\lambda}
 \sum_{\substack{u_1+\cdots+u_\ell=qn\\
                    p u_j+\lambda_j\equiv0\ (\mathrm{mod}\ q)}}
 \prod_{j=1}^{\ell}C_{h/q,\,\lambda_j/q}(u_j).
\end{equation}
For $n\ge1$, the strict excursion number is
\begin{equation}\label{eqmainstrict}
\displaystyle
 T_n^{(p,q)}=
 \sum_{\lambda\vdash q}\frac{p^\ell}{z_\lambda}
 \sum_{\substack{d_1+\cdots+d_\ell=pn-1\\
                    q d_j+\lambda_j\equiv0\ (\mathrm{mod}\ p)}}
 \prod_{j=1}^{\ell}C_{h/p,\,\lambda_j/p}(d_j).
\end{equation}
All Fuss--Catalan factors have positive rational parameters, nonnegative integer indices, and integer upper binomial arguments. Every summand in~\eqref{eqmainstrict} is positive.
\end{theorem}

\noindent\emph{Symmetry.} Reading a returning walk backwards and changing the sign of each increment gives a walk with jumps $+q,-p$ and the same heights in reverse order. Thus both weak and strict excursion counts are preserved:
\[
 E_n^{(p,q)}=E_n^{(q,p)},\qquad T_n^{(p,q)}=T_n^{(q,p)}.
\]
Applying Theorem~\ref{thmmain} with $p$ and $q$ exchanged gives sums over partitions of $p$. After normalization, we may therefore use partitions of $\min(p,q)$. Each such partition has at most $\min(p,q)$ parts, so every product has at most that many factors. Some partitions may contribute no terms at a given path size; the shifts in Corollary~\ref{corshifted} determine which can contribute.

In the displayed formulas, first choose a partition $\lambda$ of $q$, then one nonnegative index for each part. In the proof, $u_j$ is the number of upward steps in the unrestricted walk associated with part $\lambda_j$ in the weak formula, while $d_j$ is the number of downward steps in the corresponding walk for the strict formula. The upward-step indices must sum to $qn$; the downward-step indices must sum to $pn-1$. In addition, $pu_j+\lambda_j$ must be divisible by $q$ in the weak formula, and $qd_j+\lambda_j$ must be divisible by $p$ in the strict formula. These divisibility conditions ensure that the number of steps of the other type is an integer. Each allowed combination of indices contributes a product of $\ell$ Fuss--Catalan coefficients, multiplied by $(-1)^{\ell-1}q^\ell/z_\lambda$ in the weak formula or $p^\ell/z_\lambda$ in the strict formula. The partitions depend only on $q$; changing $n$ changes which combinations of indices satisfy these conditions.

For example, take $(p,q,n)=(4,3,1)$ and the partition $\lambda=(2,1)$ in the strict formula. Here $\ell=2$ and $z_\lambda=2$, so the weight is $4^2/2=8$. The conditions $d_1+d_2=3$, $3d_1+2\equiv0\pmod4$, and $3d_2+1\equiv0\pmod4$ allow only $(d_1,d_2)=(2,1)$. Its contribution is
\[
 8\,C_{7/4,1/2}(2)\,C_{7/4,1/4}(1)
 =8\cdot\frac34\cdot\frac14=\frac32.
\]
This fractional value is an algebraic contribution to the total count; it does not count a separate class of paths. Adding all three partition contributions gives the integer excursion count $5$.

For a chosen partition $\lambda=(\lambda_1,\ldots,\lambda_\ell)$, each part $\lambda_j$ supplies one Fuss--Catalan coefficient
\[
 C_{h/q,\lambda_j/q}(u_j)
\]
to a product in the formula for $E_n^{(p,q)}$. These individual coefficients are the factors in that product. The products in the formula for $T_n^{(p,q)}$ instead contain the coefficients $C_{h/p,\lambda_j/p}(d_j)$. In the proof, the coefficients used for $E_n^{(p,q)}$ come from unrestricted walks ending at $-\lambda_j$, and those used for $T_n^{(p,q)}$ from unrestricted walks ending at $+\lambda_j$. Once the final height is fixed, knowing the number of upward steps determines the number of downward steps, and conversely. The congruences ensure that the calculated number of steps is an integer.

Write $K=u+d$ for the total number of steps in a walk. For example, take jumps $+4$ and $-3$ and the part $\lambda_j=1$ in the formula for $E_n^{(p,q)}$. The associated walks must end at $-1$. With $u$ upward and $d$ downward steps, their final height is $4u-3d$, so
\[
 4u-3d=-1,\qquad d=\frac{4u+1}{3}.
\]
Choosing $u=1$ would require $d=5/3$ downward steps, which is impossible. Choosing $u=2$ gives $d=3$, hence $K=5$, as illustrated in Figure~\ref{figcongruencesteps}. The allowed values are $u=2,5,8,\ldots$. Thus the condition $4u+1\equiv0\pmod3$ simply says that $4u+1$ must be divisible by $3$ for the walk to exist.

\begin{figure}[ht]
\centering
\definecolor{countvalid}{HTML}{156F79}
\definecolor{countinvalid}{HTML}{B54F4B}
\definecolor{countgrid}{HTML}{DDE2E6}
{\bfseries Jumps $+4,-3$, final height $-1$}\par\smallskip
\(4u-3d=-1\quad\Longrightarrow\quad d=(4u+1)/3\)\par\medskip
\begin{minipage}[t]{0.47\linewidth}
\centering
{\small\bfseries One allowed index: $u=2$}\par\smallskip
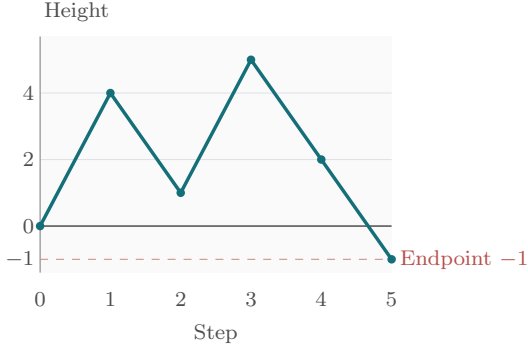
\begin{tikzpicture}[x=0.93cm,y=0.44cm,font=\scriptsize]
\path[use as bounding box] (-0.65,-3.9) rectangle (6.8,7.05);
\fill[black!2] (0,-1.4) rectangle (5,5.7);
\foreach \y in {2,4} {\draw[countgrid,line width=0.3pt] (0,\y)--(5,\y);}
\draw[black!50,line width=0.4pt] (0,-1.4)--(0,5.7);
\draw[black!65,line width=0.6pt] (0,0)--(5,0);
\draw[countinvalid!65,dashed] (0,-1)--(5,-1);
\foreach \y in {-1,0,2,4} {\node[left,inner sep=2pt,text=black!70] at (0,\y) {$\y$};}
\foreach \x in {0,1,2,3,4,5} {\node[below,text=black!70] at (\x,-1.65) {$\x$};}
\node[text=black!70] at (2.5,-3.25) {Step};
\node[anchor=west,text=black!70] at (-0.1,6.45) {Height};
\draw[countvalid,line width=1.3pt,line join=round] plot coordinates {(0,0) (1,4) (2,1) (3,5) (4,2) (5,-1)};
\foreach \p in {(0,0),(1,4),(2,1),(3,5),(4,2),(5,-1)} {\fill[countvalid] \p circle[radius=1.6pt];}
\node[right,inner sep=3pt,text=countinvalid] at (5,-1) {Endpoint $-1$};
\end{tikzpicture}\par\smallskip
{\small Jumps: \(\textcolor{countvalid}{+4},\ \textcolor{countinvalid}{-3},\ \textcolor{countvalid}{+4},\ \textcolor{countinvalid}{-3},\ \textcolor{countinvalid}{-3}\)}\par\smallskip
{\small Two upward steps, three downward steps.}\par
{\small Total length $K=2+3=5$.}
\end{minipage}\hfill
\begin{minipage}[t]{0.47\linewidth}
\centering
{\small\bfseries Which indices give whole steps?}\par\smallskip
\renewcommand{\arraystretch}{1.5}
\begin{tabular}{ccc}
\toprule
$u$ & Required $d$ & Walk possible?\\
\midrule
$0$ & $1/3$ & \textcolor{countinvalid}{No}\\
$1$ & $5/3$ & \textcolor{countinvalid}{No}\\
$2$ & $3$ & \textcolor{countvalid}{Yes}\\
$3$ & $13/3$ & \textcolor{countinvalid}{No}\\
$4$ & $17/3$ & \textcolor{countinvalid}{No}\\
$5$ & $7$ & \textcolor{countvalid}{Yes}\\
\bottomrule
\end{tabular}\par\medskip
{\small The allowed indices are $u=2,5,8,\ldots$}\par\smallskip
{\small Equivalently, $4u+1\equiv0\pmod3$.}
\end{minipage}\par\medskip
\fbox{\begin{minipage}{0.9\linewidth}\centering\small
For $u=2$, the binomial coefficient is $\binom{5}{2}=10$: choose which two of the five steps are upward. Each of these ten walks ends at $-1$.
\end{minipage}}
\caption{The congruence selects integer step counts. The displayed walk is one of the ten unrestricted walks with two upward and three downward steps. Their intermediate heights are unrestricted. Only the binomial coefficient counts these walks; the Fuss--Catalan factor also contains a multiplier.}
\label{figcongruencesteps}
\end{figure}

In general, a walk with $u_j$ upward steps gains height $pu_j$. To end at $-\lambda_j$, its downward steps must remove this height plus another $\lambda_j$. Their number is $(pu_j+\lambda_j)/q$, which explains the congruence in the formula for $E_n^{(p,q)}$:
\[
 pu_j+\lambda_j\equiv0\pmod q.
\]
For the auxiliary walk associated with part $\lambda_j$, write $K_j$ for its total number of steps. We obtain
\[
 K_j-u_j=\frac{pu_j+\lambda_j}{q},
 \qquad K_j=\frac{hu_j+\lambda_j}{q}.
\]
For the formula for $T_n^{(p,q)}$, we instead specify $d_j$ downward steps. These lose height $qd_j$, so reaching $+\lambda_j$ requires $(qd_j+\lambda_j)/p$ upward steps. This gives its congruence
\[
 qd_j+\lambda_j\equiv0\pmod p
\]
and
\[
 K_j-d_j=\frac{qd_j+\lambda_j}{p},
 \qquad K_j=\frac{hd_j+\lambda_j}{p}.
\]
Every walk ending at the given height satisfies the corresponding congruence, which excludes only indices that would require fractional numbers of steps.

Substituting the parameters of these coefficients into their definition gives
\[
 \begin{aligned}
 C_{h/q,\lambda_j/q}(u_j)&=\frac{\lambda_j}{qK_j}\binom{K_j}{u_j},\\
 C_{h/p,\lambda_j/p}(d_j)&=\frac{\lambda_j}{pK_j}\binom{K_j}{d_j}.
 \end{aligned}
\]
The number on top of each binomial coefficient is $K_j$, the total number of step positions. The first binomial coefficient chooses which $u_j$ positions contain upward steps; the second chooses which $d_j$ positions contain downward steps. The congruences make $K_j$ an integer, and $\lambda_j>0$ makes the other step count positive, so these are ordinary binomial coefficients. Their multipliers can be fractional: in the example, $\binom52=10$ counts the walks, while $C_{7/3,1/3}(2)=(1/15)\cdot10=2/3$.

When $n=0$ in~\eqref{eqmainweak}, all $u_j$ are zero. The congruences then require every part $\lambda_j$ to be divisible by $q$. Since the parts sum to $q$, only $\lambda=(q)$ contributes, giving $E_0=1$.

\section{Splitting walks into three pieces}\label{secproof}
We start with the generating function for unrestricted walks. Cutting a walk at the first and last visits to its lowest height gives an initial, a middle and a final piece. By using this decomposition, we can express the generating function as a product of three series, giving the classical factorization~\cite{Gessel1980,Xin,BLW}. From this product we derive identities that recover $E_n^{(p,q)}$ and $T_n^{(p,q)}$ by extracting endpoint coefficients at $-q$ and $+q$, respectively. Expanding the corresponding exponentials produces sums indexed by partitions of $q$. Finally, we express the walk counts divided by their lengths in terms of Fuss--Catalan coefficients to obtain the two product formulas. Subsection~\ref{subsecreindexing} then incorporates the congruence restrictions into a new choice of indices.

\subsection{The positive, zero and negative factors}
Let $W(t,x)$ be the generating function for unrestricted walks, with $t$ marking the number of steps and $x$ their final height. Put $P(x)=x^p+x^{-q}$. In $P(x)^K$, choosing one of the two terms from each factor specifies an ordered sequence of $K$ steps; multiplying their weights adds the height changes. Thus $[x^r]P(x)^K$ counts the unrestricted walks of length $K$ ending at $r$. Summing over all lengths, including the empty walk, gives
\[
 W(t,x)=1+tP(x)+t^2P(x)^2+\cdots=\frac{1}{1-tP(x)}.
\]
For example, with jumps $+3$ and $-2$,
\[
 P(x)^2=x^6+2x+x^{-4}.
\]
The coefficient $2$ of $x$ records the two walks $(+3,-2)$ and $(-2,+3)$, both of which end at height $1$.

The formal identity $-\log(1-w)=\sum_{K\ge1}w^K/K$ now gives
\[
 \log W(t,x)=\sum_{K\ge1}\frac{t^K}{K}P(x)^K.
\]
Consequently, each length-$K$ walk contributes weight $1/K$ to the logarithm. Both $W$ and $\log W$ belong to $\Q[x,x^{-1}][[t]]$: they are series in $t$ whose coefficients are Laurent polynomials in $x$.

By taking the logarithm, we turn a product decomposition into a sum. We separate its terms according to whether the exponent of $x$ is negative, zero or positive:
\begin{equation}\label{eqsplit}
 \log W=L_-(t,x)+L_0(t)+L_+(t,x).
\end{equation}
Exponentiating gives
\[
 W=\exp(L_-+L_0+L_+)=\exp(L_-)\exp(L_0)\exp(L_+).
\]
Each factor has constant term one in $t$. In $\exp(L_-)$, every nonconstant term has a negative power of $x$. In $\exp(L_+)$, every nonconstant term has a positive power of $x$, while $\exp(L_0)$ is independent of $x$. These conditions determine the factorization uniquely: the logarithm of each factor has the same restriction on its powers of $x$, and so must equal the corresponding part of $\log W$.

For a walk of length $K$, let $\eta$ be its minimum height, reached first at time $i$ and last at time $j$. Cutting at these times gives pieces of lengths $i$, $j-i$ and $K-j$. Note that any of these pieces may be empty. A nonempty initial piece ends at $\eta$ and has all its earlier heights strictly above $\eta$. We can shift the middle and final pieces vertically so that they start at zero. The middle piece is a weak excursion: it may revisit zero between its endpoints. The final piece, if nonempty, has positive height at every noninitial time and is therefore a positive meander. Let $M_{>}(t,x)$ be the generating function for positive meanders, including the empty walk, with $t$ marking the number of steps and $x$ their final height. Figure~\ref{figminimumsplit} illustrates this decomposition.

\begin{figure}[tbp]
\centering
\definecolor{splitinitial}{HTML}{B54F4B}
\definecolor{splitmiddle}{HTML}{B97514}
\definecolor{splitfinal}{HTML}{156F79}
\definecolor{splitguide}{HTML}{8C969F}
\noindent\makebox[\linewidth][l]{\small\bfseries (a) Unrestricted walk\hfill\normalfont Jumps $+4,-3$}\par\smallskip
\begin{tikzpicture}[x=0.55cm,y=0.28cm,font=\scriptsize]
\path[use as bounding box] (-3,-8.9) rectangle (25.5,8.8);
\draw[black!25,line width=0.4pt] (0,-4.4)--(0,6.2);
\draw[black!25,line width=0.4pt] (0,-6.6)--(25,-6.6);
\draw[black!30,line width=0.45pt] (0,0)--(25.3,0);
\draw[splitguide,dashed,line width=0.6pt] (0,-4)--(25.3,-4);
\foreach \x in {6,20} {\draw[splitguide!65,densely dashed,line width=0.45pt] (\x,-4)--(\x,6.8);}
\foreach \y in {0,4} {\draw[black!40] (-0.12,\y)--(0,\y);\node[left,inner sep=3pt,text=black!70] at (-0.12,\y) {$\y$};}
\node[left,inner sep=3pt,text=black!70] at (-0.12,-4) {$\eta=-4$};
\node[rotate=90,text=black!70] at (-2.65,1) {Height};
\foreach \x/\lab in {0/0,6/{i=6},13/13,20/{j=20},25/25} {\draw[black!40] (\x,-6.6)--(\x,-6.85);\node[below,text=black!70,inner sep=2pt] at (\x,-6.85) {$\lab$};}
\node[text=black!70] at (12.5,-8.45) {Step};
\node[text=splitinitial] at (3,7.6) {A\quad Initial piece};
\node[text=splitmiddle] at (13,7.6) {B\quad Middle piece};
\node[text=splitfinal] at (22.5,7.6) {C\quad Final piece};
\draw[splitinitial,line width=1.2pt,line join=round] plot coordinates {(0,0) (1,4) (2,1) (3,5) (4,2) (5,-1) (6,-4)};
\draw[splitmiddle,line width=1.2pt,line join=round] plot coordinates {(6,-4) (7,0) (8,-3) (9,1) (10,-2) (11,2) (12,-1) (13,-4) (14,0) (15,4) (16,1) (17,-2) (18,2) (19,-1) (20,-4)};
\draw[splitfinal,line width=1.2pt,line join=round] plot coordinates {(20,-4) (21,0) (22,-3) (23,1) (24,-2) (25,2)};
\foreach \p in {(0,0),(1,4),(2,1),(3,5),(4,2),(5,-1),(6,-4)} {\fill[splitinitial] \p circle[radius=1.25pt];}
\foreach \p in {(6,-4),(7,0),(8,-3),(9,1),(10,-2),(11,2),(12,-1),(13,-4),(14,0),(15,4),(16,1),(17,-2),(18,2),(19,-1),(20,-4)} {\fill[splitmiddle] \p circle[radius=1.25pt];}
\foreach \p in {(20,-4),(21,0),(22,-3),(23,1),(24,-2),(25,2)} {\fill[splitfinal] \p circle[radius=1.25pt];}
\foreach \x in {6,20} {\draw[black!75,line width=0.8pt,fill=white] (\x,-4) circle[radius=2.8pt];}
\draw[splitguide,line width=0.45pt] (5.75,-4.25)--(3.8,-5.3);
\draw[splitguide,line width=0.45pt] (20.25,-4.25)--(22.2,-5.3);
\node[below,inner sep=1pt] at (3.5,-5.3) {First minimum};
\node[below,inner sep=1pt] at (22.5,-5.3) {Last minimum};
\end{tikzpicture}\par\medskip
\noindent\makebox[\linewidth][l]{\small\bfseries (b) The three factors}\par\medskip
\begin{minipage}[t]{0.29\linewidth}
\centering
{\small\color{splitinitial} A\quad Initial piece}\par\smallskip
\begin{tikzpicture}[x=0.55cm,y=0.22cm,font=\scriptsize]
\path[use as bounding box] (-1.5,-7.2) rectangle (6.35,10);
\draw[black!25,line width=0.4pt] (0,-4.4)--(0,9);
\draw[black!25,line width=0.4pt] (0,-5)--(6.2,-5);
\draw[black!30,line width=0.45pt] (0,0)--(6.2,0);
\draw[splitguide,dashed,line width=0.55pt] (0,-4)--(6.2,-4);
\foreach \y in {-4,0,4,8} {\node[left,inner sep=3pt,text=black!70] at (0,\y) {$\y$};}
\foreach \x in {0,6} {\draw[black!40] (\x,-5)--(\x,-5.25);\node[below,inner sep=2pt,text=black!70] at (\x,-5.25) {$\x$};}
\node[text=black!70] at (3,-7) {Local step};
\draw[splitinitial,line width=1.1pt,line join=round] plot coordinates {(0,0) (1,4) (2,1) (3,5) (4,2) (5,-1) (6,-4)};
\foreach \p in {(0,0),(1,4),(2,1),(3,5),(4,2),(5,-1),(6,-4)} {\fill[splitinitial] \p circle[radius=1.2pt];}
\draw[splitinitial,line width=0.8pt,fill=white] (6,-4) circle[radius=2.4pt];
\end{tikzpicture}\par
{\scriptsize First minimum\\at its endpoint}\par\smallskip
{\footnotesize $\exp L_-(t,x)$}
\end{minipage}\hfill
\begin{minipage}[t]{0.40\linewidth}
\centering
{\small\color{splitmiddle} B\quad Middle piece}\par\smallskip
\begin{tikzpicture}[x=0.40cm,y=0.22cm,font=\scriptsize]
\path[use as bounding box] (-1.5,-7.2) rectangle (14.35,10);
\draw[black!25,line width=0.4pt] (0,-4.4)--(0,9);
\draw[black!25,line width=0.4pt] (0,-5)--(14.2,-5);
\draw[black!30,line width=0.45pt] (0,0)--(14.2,0);
\foreach \y in {-4,0,4,8} {\node[left,inner sep=3pt,text=black!70] at (0,\y) {$\y$};}
\foreach \x in {0,7,14} {\draw[black!40] (\x,-5)--(\x,-5.25);\node[below,inner sep=2pt,text=black!70] at (\x,-5.25) {$\x$};}
\node[text=black!70] at (7,-7) {Local step};
\draw[splitmiddle,line width=1.1pt,line join=round] plot coordinates {(0,0) (1,4) (2,1) (3,5) (4,2) (5,6) (6,3) (7,0) (8,4) (9,8) (10,5) (11,2) (12,6) (13,3) (14,0)};
\foreach \p in {(0,0),(1,4),(2,1),(3,5),(4,2),(5,6),(6,3),(7,0),(8,4),(9,8),(10,5),(11,2),(12,6),(13,3),(14,0)} {\fill[splitmiddle] \p circle[radius=1.2pt];}
\foreach \x in {0,7,14} {\draw[splitmiddle,line width=0.8pt,fill=white] (\x,0) circle[radius=2.4pt];}
\draw[splitguide,line width=0.45pt] (7,-0.5)--(7,-2.6);
\node[below,inner sep=1pt] at (7,-2.7) {Interior return allowed};
\end{tikzpicture}\par
{\scriptsize Weak excursion\\shifted by $-\eta$}\par\smallskip
{\footnotesize $E(t^h)=\exp L_0(t)$}
\end{minipage}\hfill
\begin{minipage}[t]{0.29\linewidth}
\centering
{\small\color{splitfinal} C\quad Final piece}\par\smallskip
\begin{tikzpicture}[x=0.60cm,y=0.22cm,font=\scriptsize]
\path[use as bounding box] (-1.3,-7.2) rectangle (5.35,10);
\draw[black!25,line width=0.4pt] (0,-4.4)--(0,9);
\draw[black!25,line width=0.4pt] (0,-5)--(5.2,-5);
\draw[black!30,line width=0.45pt] (0,0)--(5.2,0);
\foreach \y in {-4,0,4,8} {\node[left,inner sep=3pt,text=black!70] at (0,\y) {$\y$};}
\foreach \x in {0,5} {\draw[black!40] (\x,-5)--(\x,-5.25);\node[below,inner sep=2pt,text=black!70] at (\x,-5.25) {$\x$};}
\node[text=black!70] at (2.5,-7) {Local step};
\draw[splitfinal,line width=1.1pt,line join=round] plot coordinates {(0,0) (1,4) (2,1) (3,5) (4,2) (5,6)};
\foreach \p in {(0,0),(1,4),(2,1),(3,5),(4,2),(5,6)} {\fill[splitfinal] \p circle[radius=1.2pt];}
\draw[splitfinal,line width=0.8pt,fill=white] (0,0) circle[radius=2.4pt];
\end{tikzpicture}\par
{\scriptsize Positive meander\\shifted by $-\eta$}\par\smallskip
{\footnotesize $M_{>}(t,x)=\exp L_+(t,x)$}
\end{minipage}\par\smallskip
\caption{Cutting a walk at its first and last visits to its global minimum $\eta$. In (a), $\eta=-4$, $i=6$ and $j=20$. In (b), the middle and final pieces are shifted vertically by $-\eta$ so that they start at zero. The initial piece ends at its first minimum; the middle piece is a weak excursion, with an interior return to zero; and the final piece is a positive meander. The colours identify the same pieces in both panels.}
\label{figminimumsplit}
\end{figure}

Conversely, one such initial piece, one weak excursion and one positive meander determine a unique unrestricted walk by concatenation, shifting the last two pieces to start at height $\eta$. The first and last minimum visits occur precisely at the joins. For given lengths and changes in height, every choice of an initial piece can be combined with every choice of a middle piece and every choice of a final piece. The number of combinations is therefore the product of the three counts. Concatenation adds their lengths and endpoint changes, so multiplying the three generating functions counts these combinations. Since the decomposition is unique, each unrestricted walk is counted exactly once. Each series has constant term one in $t$, and its nonempty paths have negative, zero or positive endpoint changes, respectively. By the uniqueness of the preceding factorization, these series are $\exp L_-$, $\exp L_0$ and $\exp L_+$, in that order. In particular,
\begin{equation}\label{eqpositivefactor}
 E(t^h)=\exp L_0(t),\qquad
 M_{>}(t,x)=\exp L_+(t,x).
\end{equation}
Here $E(t^h)$ records step length rather than excursion size, since a size-$n$ excursion has $hn$ steps.

Cutting instead at the first and last visits to the maximum gives the same three factors after reversing the appropriate pieces~\cite[Theorem~10]{BLW}. The choice between partitions of $q$ and partitions of $p$ is determined by the endpoint coefficient extracted, rather than by the choice of extremum.

Similarly, let $M_{\ge}(t,x)$ be the generating function for weak meanders, including the empty walk, with the same length and height variables. Cutting at the last visit to zero gives a weak excursion followed by a positive meander. Hence
\begin{equation}\label{eqweakfactor}
 \begin{aligned}
 M_{\ge}(t,x)&=E(t^h)M_{>}(t,x)=\exp(L_0+L_+),\\
 (1-tP(x))M_{\ge}(t,x)&=\exp(-L_-).
 \end{aligned}
\end{equation}
The last identity follows from $1-tP(x)=1/W(t,x)$ and the factorization of $W$.

For $r\ge1$, we write $L_r^{\uparrow}(t)$ and $L_r^{\downarrow}(t)$ for the logarithmic endpoint series at $+r$ and $-r$, respectively:
\begin{equation}\label{eqLr}
 \begin{aligned}
 L_r^{\uparrow}(t)&=[x^r]\log W(t,x)
 =\sum_{K\ge1}\frac{t^K}{K}[x^r]P(x)^K,\\
 L_r^{\downarrow}(t)&=[x^{-r}]\log W(t,x)
 =\sum_{K\ge1}\frac{t^K}{K}[x^{-r}]P(x)^K.
 \end{aligned}
\end{equation}
The arrows indicate the sign of the endpoint. Intermediate heights are unrestricted in both series: the coefficient of $t^K$ is the number of length-$K$ walks ending at the specified height, divided by $K$. In the preceding two-step example with jumps $+3,-2$, $[t^2]L_1^{\uparrow}=2/2=1$; both step orders contribute. Grouping by endpoint gives
\[
 L_+=\sum_{r\ge1}L_r^{\uparrow}(t)x^r,\qquad
 L_-=\sum_{r\ge1}L_r^{\downarrow}(t)x^{-r}.
\]

\begin{lemma}\label{lemendpoints}
Let $y$ record the absolute value of the endpoint height: $y^r$ corresponds to endpoint $+r$ in the first exponential and endpoint $-r$ in the second. The strict and weak counts satisfy
\begin{align}
 T_n^{(p,q)}&=[t^{hn-1}y^q]\exp\!\left(\sum_{r=1}^qL_r^{\uparrow}(t)y^r\right)
 \quad(n\ge1),\label{eqendstrict}\\
 E_n^{(p,q)}&=-[t^{hn+1}y^q]\exp\!\left(-\sum_{r=1}^qL_r^{\downarrow}(t)y^r\right)
 \quad(n\ge0).\label{eqendweak}
\end{align}
\end{lemma}
\begin{proof}
A nonempty strict excursion must end with a downward step $-q$: a final upward step would start at height $-p$, contrary to positivity. Its penultimate height is therefore $q$. Deleting the final step gives a positive meander of length $hn-1$ ending at $q$. Conversely, appending a $-q$ step to any such meander gives a strict excursion of length $hn$. Using~\eqref{eqpositivefactor}, we obtain
\[
 T_n^{(p,q)}=[t^{hn-1}x^q]M_{>}(t,x)
 =[t^{hn-1}x^q]\exp L_+(t,x).
\]
Replacing $x$ by $y$ gives the identity for $T_n^{(p,q)}$ in~\eqref{eqendstrict}.

For $E_n^{(p,q)}$, use~\eqref{eqweakfactor}:
\[
 (1-tx^p-tx^{-q})M_{\ge}(t,x)=\exp(-L_-(t,x)).
\]
Every weak meander ends at a nonnegative height, so $M_{\ge}(t,x)$ contains only terms $t^K x^r$ with $r\ge0$. In particular, the terms $M_{\ge}$ and $-tx^pM_{\ge}$ cannot contribute to the coefficient of $x^{-q}$. In the remaining term, multiplication by $-tx^{-q}$ changes $t^K x^r$ to $-t^{K+1}x^{r-q}$. The exponent of $x$ is $-q$ exactly when $r=0$. Thus only the coefficient of $x^0$ in $M_{\ge}$ contributes. Since a weak meander ending at zero is a weak excursion, this coefficient is $E(t^h)$. Hence
\[
 [x^{-q}]\exp(-L_-(t,x))=-tE(t^h).
\]
Taking the coefficient of $t^{hn+1}$ gives $-E_n^{(p,q)}$, so
\[
 E_n^{(p,q)}=-[t^{hn+1}x^{-q}]\exp(-L_-(t,x)).
\]
This explains both the leading minus sign and the extra power of $t$. Substituting $y=x^{-1}$ gives the identity for $E_n^{(p,q)}$ in~\eqref{eqendweak}.

In both exponentials, terms with $r>q$ cannot contribute to the coefficient of $y^q$: multiplying powers of $y$ adds their exponents, and all the exponents being added are nonnegative.
\end{proof}

\subsection{Fuss--Catalan expressions for the \texorpdfstring{$L_r$}{Lr} series}
Fix $r\ge1$. For endpoint $-r$, let $u$ and $d$ be the numbers of upward and downward steps. The endpoint equation $pu-qd=-r$ gives
\[
 d=\frac{pu+r}{q},\qquad K=u+d=\frac{hu+r}{q}.
\]
The congruence $pu+r\equiv0\pmod q$ is precisely the condition that $d$ is an integer. For each allowed $u\ge0$, choosing the $u$ upward-step positions among the $K$ positions gives $\binom{K}{u}$ unrestricted walks. Their contribution to the logarithm is $\binom{K}{u}/K$. Since $K=(h/q)u+r/q$, the definition of the Fuss--Catalan coefficients gives
\[
 \frac1K\binom{K}{u}=\frac qr C_{h/q,r/q}(u).
\]
For a fixed $u$, both $d$ and $K$ are determined. We obtain $L_r^{\downarrow}(t)$ by adding the contributions of these different step counts:
\begin{equation}\label{eqLdownFC}
 L_r^{\downarrow}(t)
 =\frac qr\sum_{\substack{u\ge0\\pu+r\equiv0\pmod q}}
 C_{h/q,r/q}(u)\,t^{(hu+r)/q}.
\end{equation}

For endpoint $+r$, writing the length in terms of $u$ gives $K=(h/q)u-r/q$. Comparing this with $\alpha u+\beta$ would give $\beta=-r/q$. If one uses $d$, the number of downward steps, one obtains
\[
 u=\frac{qd+r}{p},\qquad K=u+d=\frac{hd+r}{p}.
\]
Now $K=(h/p)d+r/p$, giving the positive second parameter $r/p$. The congruence $qd+r\equiv0\pmod p$ ensures that $u$ is an integer. Choosing the $d$ downward-step positions gives the weighted count
\[
 \frac1K\binom{K}{d}=\frac pr C_{h/p,r/p}(d).
\]
Therefore
\begin{equation}\label{eqLupFC}
 L_r^{\uparrow}(t)
 =\frac pr\sum_{\substack{d\ge0\\qd+r\equiv0\pmod p}}
 C_{h/p,r/p}(d)\,t^{(hd+r)/p}.
\end{equation}
In both sums, the exponent of $t$ is $K$, the total number of steps. Each endpoint series is a sum over the possible step counts for one final height. In Subsection~\ref{subsecpartitionexpansion}, we substitute these series into the exponentials of Lemma~\ref{lemendpoints}. A partition of $q$ selects a product of endpoint series, and multiplying their sums gives a convolution over their step counts.

Lagrange inversion verifies the coefficient identity in~\eqref{eqfcintro}. Set $G(x)=F_\alpha(x)-1$. Then $G=x(1+G)^\alpha$, and for $m\ge1$,
\[
 [x^m]F_\alpha(x)^\beta
 =\frac{\beta}{m}[w^{m-1}](1+w)^{\alpha m+\beta-1}
 =\frac{\beta}{\alpha m+\beta}\binom{\alpha m+\beta}{m}.
\]
The constant coefficient is one. When $\alpha$ and $\beta$ are rational, we expand $(1+w)^{\alpha m+\beta-1}$ using the generalized binomial theorem. Its coefficient of $w^{m-1}$ is still the binomial coefficient displayed above~\cite{Gessel2016}.

\subsection{Completing the partition expansion}\label{subsecpartitionexpansion}
We now substitute the endpoint series into the exponentials of Lemma~\ref{lemendpoints}. Expanding these exponentials produces products of series for endpoint heights that add to $q$.
For arbitrary series $B_1,\ldots,B_q$ and $\varepsilon\in\{+1,-1\}$, expand the exponential as
\[
 \exp\!\left(\varepsilon\sum_{r=1}^qB_ry^r\right)
 =\prod_{r=1}^q\left(\sum_{k\ge0}\frac{\varepsilon^kB_r^k y^{rk}}{k!}\right).
\]
Choosing $m_r(\lambda)$ copies of $B_ry^r$ from the $r$th factor contributes exponent $r m_r(\lambda)$ and denominator $m_r(\lambda)!$. To obtain $y^q$, these exponents must sum to $q$; the choices therefore correspond to partitions $\lambda\vdash q$. There are $\ell$ copies in total, giving the sign $\varepsilon^\ell$. Hence
\begin{equation}\label{eqexponentialpartitions}
 [y^q]\exp\!\left(\varepsilon\sum_{r=1}^qB_ry^r\right)
 =\sum_{\lambda\vdash q}\frac{\varepsilon^{\ell}}{\prod_r m_r(\lambda)!}
 \prod_{j=1}^{\ell}B_{\lambda_j}.
\end{equation}
For example, when $q=3$, this coefficient is
\[
 \varepsilon B_3+\varepsilon^2 B_2B_1+\frac{\varepsilon^3}{3!}B_1^3,
\]
corresponding to $(3)$, $(2,1)$ and $(1,1,1)$.

For $E_n^{(p,q)}$, insert~\eqref{eqLdownFC} into~\eqref{eqendweak}. For a fixed partition, multiplying the $\ell$ endpoint series adds their step lengths:
\[
 \sum_j\frac{h u_j+\lambda_j}{q}
 =\frac hq\sum_ju_j+1,
\]
because $\sum_j\lambda_j=q$. The extracted length is $hn+1$, so $\sum_j u_j=qn$. Each selected index retains its congruence $pu_j+\lambda_j\equiv0\pmod q$. Multiplying the factors $q/\lambda_j$ gives $q^\ell/\prod_j\lambda_j$; the multiplicity factorials turn this into $q^\ell/z_\lambda$. Here $\varepsilon=-1$, and the leading minus sign in the identity for $E_n^{(p,q)}$ gives $-(-1)^\ell=(-1)^{\ell-1}$. This proves~\eqref{eqmainweak}.

For $T_n^{(p,q)}$, insert~\eqref{eqLupFC} into~\eqref{eqendstrict}. The combined length is
\[
 \sum_j\frac{h d_j+\lambda_j}{p}
 =\frac hp\sum_jd_j+\frac qp.
\]
Since $h=p+q$, requiring length $hn-1$ is equivalent to $\sum_jd_j=pn-1$. Each selected index satisfies $qd_j+\lambda_j\equiv0\pmod p$. Every endpoint series supplies one multiplier $p/\lambda_j$, so the $\ell$ factors give
\[
 \frac{1}{\prod_r m_r(\lambda)!}\prod_{j=1}^{\ell}\frac p{\lambda_j}
 =\frac{p^\ell}{(\prod_j\lambda_j)(\prod_r m_r(\lambda)!)}
 =\frac{p^\ell}{z_\lambda}.
\]
Here $\varepsilon=1$, so all signs are positive. Every selected step length in both calculations is positive, since the step counts are nonnegative and the partition parts are positive. Thus~\eqref{eqmainstrict} follows, completing the proof of Theorem~\ref{thmmain}.

\subsection{Rewriting the sums}\label{subsecreindexing}
The congruences in Theorem~\ref{thmmain} select indices with regular spacing. We write the number $u$ of upward steps as a multiple of $q$ plus its remainder. For jumps $+4$ and $-3$ and partition part $1$, the possible numbers of upward steps are $2,5,8,\ldots$. Writing
\[
 u=3v+2,\qquad v=0,1,2,\ldots,
\]
lists these values exactly once and incorporates the divisibility condition into the parametrization. For $\lambda=(1,1,1)$ in the formula for $E_n^{(4,3)}$, the condition $u_1+u_2+u_3=3n$ becomes
\[
 (3v_1+2)+(3v_2+2)+(3v_3+2)=3n,
 \qquad v_1+v_2+v_3=n-2.
\]
Each of the three indices includes the fixed remainder $2$, so the remainders contribute $6$ to the total $3n$. Subtracting $6$ and dividing by $3$ gives $v_1+v_2+v_3=n-2$. We now give the corresponding remainders and index totals for both formulas.

Write $r$ for a partition part, so $1\le r\le q$. We describe both endpoint types using upward-step counts, with $\sigma=+1$ for the factors in the formula for $E_n^{(p,q)}$ and $\sigma=-1$ for those in the formula for $T_n^{(p,q)}$:
\begin{center}
\begin{tabular}{cccc}
\toprule
$\sigma$ & Count formula & Endpoint & Step-count equation\\
\midrule
$+1$ & $E_n^{(p,q)}$ & $-r$ & $pu+r=qd$\\
$-1$ & $T_n^{(p,q)}$ & $+r$ & $pu-r=qd$\\
\bottomrule
\end{tabular}
\end{center}
The superscript follows the sign of $r$ in the congruence; the endpoint has the opposite sign. Define
\begin{equation}\label{eqresidues}
\begin{gathered}
 0\le k_r^\sigma<q,\qquad pk_r^\sigma+\sigma r\equiv0\pmod q,\\
 a_r^\sigma=\frac{hk_r^\sigma+\sigma r}{q},\qquad
 \delta_\lambda^\sigma=\frac1q\sum_{j=1}^{\ell}k_{\lambda_j}^\sigma.
\end{gathered}
\end{equation}
Coprimality of $p$ and $q$ makes $k_r^\sigma$ unique. This number is the remainder when the number $u$ of upward steps is divided by $q$. Every nonnegative $u$ satisfying the congruence can therefore be written as
\[
 u=qv+k_r^\sigma,\qquad v=0,1,2,\ldots.
\]
The remainder $k_r^\sigma$ is the same for all values of $u$ satisfying the congruence; only the quotient $v$ changes. In the example with $p=4$, $q=3$ and final height $-1$, the values $2,5,8,\ldots$ all leave remainder $2$ upon division by $3$. Thus $k_1^+=2$ and $u=3v+2$.

Substituting this expression for $u$ into the length formula gives
\[
 K=hv+a_r^\sigma.
\]
Here $a_r^\sigma$ is the number added to $hv$ to obtain the total length. For the same example, $d=4v+3$ and $K=u+d=7v+5$, so $a_1^+=5$. The quantity $\delta_\lambda^\sigma$ gives the shift in the sum of the new indices caused by their remainders. For instance,
\[
 \delta_{(1,1,1)}^+=\frac{2+2+2}{3}=2,
\]
which accounts for the total $n-2$ in the preceding example. When $q=1$, both residues are zero.

In the formula for $T_n^{(p,q)}$, a coefficient has the form $C_{h/p,r/p}(d)$, where $d$ is the number of downward steps. To express this coefficient in terms of $v$, we substitute $u=qv+k_r^-$ into $d=(pu-r)/q$. This gives $d=pv+b_r$, where
\begin{equation}\label{eqbminus}
 b_r=\frac{p k_r^- -r}{q}=a_r^- -k_r^-.
\end{equation}
The coefficient is therefore evaluated as $C_{h/p,r/p}(pv+b_r)$. For example, with $p=4$, $q=3$ and endpoint $+1$, one has $k_1^-=1$, so
\[
 u=3v+1,\qquad d=4v+1.
\]
Thus $b_1=1$, and the corresponding coefficients are $C_{7/4,1/4}(1)$, $C_{7/4,1/4}(5)$, $C_{7/4,1/4}(9)$, and so on.

For $r<q$, one has $b_r\ge0$, so every $v\ge0$ gives nonnegative step counts. For $r=q$, one has $k_q^-=0$ and $b_q=-1$, so $d=pv-1$ requires $v\ge1$. This part occurs only in the partition $(q)$.

\begin{corollary}\label{corshifted}
The shifts $\delta_\lambda^\pm$ are nonnegative integers. Formulas~\eqref{eqmainweak} and~\eqref{eqmainstrict} are equivalently
\begin{align}
 E_n^{(p,q)}&=\sum_{\lambda\vdash q}\frac{(-1)^{\ell-1}q^\ell}{z_\lambda}
 \sum_{v_1+\cdots+v_\ell=n-\delta_\lambda^+}
 \prod_{j=1}^{\ell}C_{h/q,\,\lambda_j/q}(qv_j+k_{\lambda_j}^+),\label{eqshiftedfcweak}\\
 T_n^{(p,q)}&=\sum_{\lambda\vdash q}\frac{p^\ell}{z_\lambda}
 \sum_{v_1+\cdots+v_\ell=n-\delta_\lambda^-}
 \prod_{j=1}^{\ell}C_{h/p,\,\lambda_j/p}(pv_j+b_{\lambda_j}),\label{eqshiftedfcstrict}
\end{align}
where the strict identity is for $n\ge1$. In the one-part strict term, $v_1=n\ge1$, so its Fuss--Catalan index $pn-1$ is nonnegative. In ordinary-binomial form these become
\begin{align}
 E_n^{(p,q)}&=\sum_{\lambda\vdash q}\frac{(-1)^{\ell-1}}{\prod_r m_r(\lambda)!}
 \sum_{\sum_jv_j=n-\delta_\lambda^+}
 \prod_{j=1}^{\ell}\frac{\binom{hv_j+a_{\lambda_j}^+}{qv_j+k_{\lambda_j}^+}}
                              {hv_j+a_{\lambda_j}^+},\label{eqbinomialweak}\\
 T_n^{(p,q)}&=\sum_{\lambda\vdash q}\frac{1}{\prod_r m_r(\lambda)!}
 \sum_{\sum_jv_j=n-\delta_\lambda^-}
 \prod_{j=1}^{\ell}\frac{\binom{hv_j+a_{\lambda_j}^-}{qv_j+k_{\lambda_j}^-}}
                              {hv_j+a_{\lambda_j}^-}.\label{eqbinomialstrict}
\end{align}
\end{corollary}

\begin{proof}
Summing the congruences that define the remainders, and using $\sum_j\lambda_j=q$, gives
\[
 p\sum_j k_{\lambda_j}^\sigma
 \equiv-\sigma\sum_j\lambda_j=-\sigma q\equiv0\pmod q.
\]
Since $p$ and $q$ are coprime, $\sum_j k_{\lambda_j}^\sigma$ is divisible by $q$. Thus $\delta_\lambda^\sigma$ is an integer; it is nonnegative because every remainder is nonnegative. The length constants satisfy
\begin{equation}\label{eqsumoffsets}
 \sum_j a_{\lambda_j}^\sigma
 =\frac{h\sum_j k_{\lambda_j}^\sigma+\sigma q}{q}
 =h\delta_\lambda^\sigma+\sigma.
\end{equation}

In the formula for $E_n^{(p,q)}$, substituting $u_j=qv_j+k_{\lambda_j}^+$ gives
\[
 \sum_j u_j=q\sum_jv_j+q\delta_\lambda^+.
\]
The condition $\sum_j u_j=qn$ is therefore equivalent to $\sum_jv_j=n-\delta_\lambda^+$. Substitution into each Fuss--Catalan coefficient gives the first shifted formula.

For endpoint $+r$, every nonnegative step-count pair has the unique representation
\[
 u=qv+k_r^-,\qquad d=pv+b_r.
\]
If $r<q$, then $pk_r^- -r$ is a multiple of $q$ greater than $-q$, so $b_r\ge0$ and every $v\ge0$ gives nonnegative counts. If $r=q$, then $k_q^-=0$ and $b_q=-1$, so $d\ge0$ requires $v\ge1$. This part occurs only in $\lambda=(q)$; here $pv-1=pn-1$ forces $v=n\ge1$.

Using $b_r=a_r^- -k_r^-$ and the relation for the length constants, we obtain
\[
 \sum_j b_{\lambda_j}
 =(h\delta_\lambda^- -1)-q\delta_\lambda^-
 =p\delta_\lambda^- -1.
\]
Hence the downward-step total in the formula for $T_n^{(p,q)}$ becomes
\[
 \sum_j d_j=p\sum_jv_j+p\delta_\lambda^- -1.
\]
Equating this to $pn-1$ gives $\sum_jv_j=n-\delta_\lambda^-$, proving the second shifted formula.

For allowed indices, substituting the definition of the Fuss--Catalan coefficients gives
\begin{align}
 \frac{q}{r}C_{h/q,r/q}(qv+k_r^+)
 &=\frac{1}{hv+a_r^+}\binom{hv+a_r^+}{qv+k_r^+},\label{eqfactorweak}\\
 \frac{p}{r}C_{h/p,r/p}(pv+b_r)
 &=\frac{1}{hv+a_r^-}\binom{hv+a_r^-}{pv+b_r}
 =\frac{1}{hv+a_r^-}\binom{hv+a_r^-}{qv+k_r^-}.\label{eqfactorstrict}
\end{align}
The last equality holds because the total length is the sum of the upward- and downward-step counts. Choosing the positions of the downward steps is equivalent to choosing the remaining positions for the upward steps.

Finally, $z_\lambda=(\prod_j\lambda_j)\prod_r m_r(\lambda)!$, so the coefficient multipliers combine with the outer weights as
\[
 \frac{q^\ell}{z_\lambda}\prod_{j=1}^{\ell}\frac{\lambda_j}{q}
 =\frac{p^\ell}{z_\lambda}\prod_{j=1}^{\ell}\frac{\lambda_j}{p}
 =\frac{1}{\prod_r m_r(\lambda)!}.
\]
The sign $(-1)^{\ell-1}$ remains in the formula for $E_n^{(p,q)}$. Substituting the binomial expressions therefore gives the two ordinary-binomial forms.
\end{proof}

\section{The product formulas}\label{secstructure}
\subsection{Integer multiples of Fuss--Catalan coefficients}

\begin{proposition}\label{propinteger}
Let $1\le r\le q$, and let $u,d$ be nonnegative upward- and downward-step counts with $|pu-qd|=r$. Put $K=u+d$. Then
\begin{equation}\label{eqintegerfactors}
 \frac rK\binom{K}{u}\in\Z_{>0}.
\end{equation}
For endpoint $-r$, this quantity equals $qC_{h/q,r/q}(u)$, as in~\eqref{eqmainweak}; for endpoint $+r$, it equals $pC_{h/p,r/p}(d)$, as in~\eqref{eqmainstrict}.
\end{proposition}
\begin{proof}
For endpoint $+r$, the step counts satisfy $pu-qd=r$, so
\[
 \frac rK\binom{K}{u}
 =\frac{pu-qd}{K}\binom{K}{u}
 =p\binom{K-1}{u-1}-q\binom{K-1}{u}.
\]
For endpoint $-r$, they satisfy $qd-pu=r$, giving
\[
 \frac rK\binom{K}{u}
 =\frac{qd-pu}{K}\binom{K}{u}
 =q\binom{K-1}{u}-p\binom{K-1}{u-1}.
\]
The right-hand sides are integers, with out-of-range binomial coefficients interpreted as zero. The left-hand sides are strictly positive because $r>0$, $K>0$ and $0\le u\le K$.
\end{proof}

The numerators $q^\ell\prod_j C_{h/q,\lambda_j/q}(u_j)$ and $p^\ell\prod_j C_{h/p,\lambda_j/p}(d_j)$ in Theorem~\ref{thmmain} are therefore products of positive integers. The integer factors have a combinatorial interpretation through marked positive meanders. For a nonempty positive meander $\pi$, let $\mu(\pi)$ be its smallest height after the start. A marked version is a pair $(\pi,j)$ with $j\in\{1,\ldots,\mu(\pi)\}$, so $\pi$ has $\mu(\pi)$ marked versions. Banderier, Lackner, and Wallner~\cite[Definition~18, Table~2 and Theorem~19]{BLW} show that $rL_r^{\uparrow}(t)$ counts marked positive meanders ending at $r$; $rL_r^{\downarrow}(t)$ has the corresponding interpretation for the reflected jumps $+q,-p$.

For jumps $+4,-3$, the two positive meanders of length four ending at height two have height sequences
\[
 (0,4,8,5,2),\qquad (0,4,1,5,2).
\]
Their smallest heights after the start are two and one, so they give $2+1=3$ marked meanders. This agrees with
\[
 [t^4]\,2L_2^{\uparrow}(t)=\frac24\binom42=3.
\]
The remaining weight $1/z_\lambda$ is rational, so a full partition contribution need not be an integer.

To express the product formulas as convolutions of integer sequences, put
\begin{align}
 A_r(v)&=\frac{r}{hv+a_r^+}\binom{hv+a_r^+}{qv+k_r^+}
       =qC_{h/q,r/q}(qv+k_r^+),\qquad v\ge0,\label{eqintegerA}\\
 D_r(v)&=\frac{r}{hv+a_r^-}\binom{hv+a_r^-}{qv+k_r^-}
       =pC_{h/p,r/p}(pv+b_r),\label{eqintegerD}
\end{align}
where the second formula holds for $v\ge0$ when $r<q$ and for $v\ge1$ when $r=q$. Set $D_q(0)=0$. Define
\[
 \mathcal A_r(z)=\sum_{v\ge0}A_r(v)z^v,\qquad
 \mathcal D_r(z)=\sum_{v\ge0}D_r(v)z^v.
\]
These series have integer coefficients, positive except for $D_q(0)$. Corollary~\ref{corshifted} becomes
\begin{align}
 E(z)&=\sum_{\lambda\vdash q}\frac{(-1)^{\ell-1}z^{\delta_\lambda^+}}{z_\lambda}
                  \prod_{j=1}^{\ell}\mathcal A_{\lambda_j}(z),\label{eqintegerseriesweak}\\
 T(z)&=\sum_{\lambda\vdash q}\frac{z^{\delta_\lambda^-}}{z_\lambda}
                  \prod_{j=1}^{\ell}\mathcal D_{\lambda_j}(z).\label{eqintegerseriesstrict}
\end{align}
The convention $D_q(0)=0$ allows every series $\mathcal D_r(z)$ to start at $v=0$, while excluding $d=-1$ when $r=q$. Each partition contribution is a product of integer sequence entries multiplied by the rational weight $1/z_\lambda$, with the additional sign $(-1)^{\ell-1}$ in the weak formula. An individual contribution may therefore be fractional, although the total sum is the integer excursion count.

\section{Explicit identities for excursion sequences}\label{secexamples}
When the downward jump has magnitude one ($q=1$), the product formulas recover Fuss--Catalan and Raney sequences. Appendix~\ref{appformulas} gives expanded binomial identities for all pairs $+6,-q$ with $1\le q<6$ and $+7,-q$ with $1\le q<7$. The identities for $E_n$ hold for $n\ge0$ and those for $T_n$ for $n\ge1$.

\subsection{Fuss--Catalan and Raney sequences for jumps \texorpdfstring{$+p$ and $-1$}{+p and -1}}
When $q=1$, the jumps are $+p$ and $-1$, and the binomial formulas reduce to
\begin{align}
 E_n^{(p,1)}&=\frac{1}{pn+1}\binom{(p+1)n}{n}\quad(n\ge0),\label{equnitweak}\\
 T_n^{(p,1)}&=\frac{p}{(p+1)n-1}\binom{(p+1)n-1}{n-1}\quad(n\ge1).\label{equnitstrict}
\end{align}
The weak counts are the usual Fuss--Catalan numbers, and the strict counts are shifted Raney numbers:
\[
 E_n^{(p,1)}=C_{p+1,1}(n),\qquad
 T_n^{(p,1)}=C_{p+1,p}(n-1)\quad(n\ge1).
\]
The generating functions satisfy $E=1+zE^{p+1}$ and $T=zE^p$. In particular,
\[
 T_n^{(p,1)}=[z^{n-1}]E(z)^p\quad(n\ge1),
\]
which explains the index $n-1$ in the Raney expression. The strict binomial formula can also be written as
\[
 T_n^{(p,1)}=\frac{1}{(p+1)n-1}\binom{(p+1)n-1}{n}
 =\frac{p}{(p+1)n-1}\binom{(p+1)n-1}{n-1}.
\]
For $p=2,3,6$, the weak sequences are \seqnum{A001764}, \seqnum{A002293}, and \seqnum{A002296}. Their strict companions are \seqnum{A006013} at index $n-1$, and \seqnum{A006632} and \seqnum{A130565} at index $n$, respectively~\cite{OEIS}. The last two OEIS entries start at index one. After normalization, original jumps $(6,3)$ and $(6,2)$ give the same counts as $(2,1)$ and $(3,1)$, respectively.

\section{Sequence identification}\label{secverification}
\subsection{OEIS correspondences and indexing}
Table~\ref{taboeis} identifies the excursion sequences in the OEIS~\cite{OEIS}, using each entry's own index. Some entries use the exchanged jumps $+q,-p$; the symmetry discussed after Theorem~\ref{thmmain} preserves both weak and strict excursion counts. In all tables, size $n$ is defined using the normalized jumps, as in the introduction.

\begin{table}[htbp]
\centering
\begin{tabular}{@{}llll@{}}
\toprule
Original jumps & Normalized pair & Weak $E_n$ & Strict $T_n$, $n\ge1$\\
\midrule
$+4,-3$ & $(4,3)$ & \seqnum{A300390}$(n)$ & ---\\
$+5,-2$ & $(5,2)$ & \seqnum{A300386}$(n)$ & ---\\
$+5,-3$ & $(5,3)$ & \seqnum{A300391}$(n)$ & ---\\
$+5,-4$ & $(5,4)$ & --- & ---\\
$+6,-1$ & $(6,1)$ & \seqnum{A002296}$(n)$ & \seqnum{A130565}$(n)$\\
$+6,-2$ & $(3,1)$ & \seqnum{A002293}$(n)$ & \seqnum{A006632}$(n)$\\
$+6,-3$ & $(2,1)$ & \seqnum{A001764}$(n)$ & \seqnum{A006013}$(n-1)$\\
$+6,-4$ & $(3,2)$ & \seqnum{A060941}$(n)$ & \seqnum{A293946}$(n)$\\
$+6,-5$ & $(6,5)$ & --- & ---\\
$+7,-1$ & $(7,1)$ & \seqnum{A007556}$(n)$ & \seqnum{A234466}$(n-1)$\\
$+7,-2$ & $(7,2)$ & \seqnum{A300387}$(n)$ & ---\\
$+7,-3$ & $(7,3)$ & --- & ---\\
$+7,-4$ & $(7,4)$ & --- & ---\\
$+7,-5$ & $(7,5)$ & --- & ---\\
$+7,-6$ & $(7,6)$ & --- & ---\\
\bottomrule
\end{tabular}
\caption{OEIS entries for the weak and strict excursion counts. The argument shown after each sequence number gives the corresponding OEIS index.}\label{taboeis}
\end{table}
These OEIS identifications agree with the entries' formulas or path interpretations, as well as their initial terms. For \seqnum{A293946}, the correspondence holds for $n\ge1$: its constant term is one, whereas $T_0=0$. For \seqnum{A006013} and \seqnum{A234466}, read the OEIS entry at index $n-1$ to obtain $T_n$. A dash means that no identification is asserted.

\section{A probability application}\label{secprobability}
Let the increments $X_j$ be independent, with $\Pr(X_j=p)=\theta$ and $\Pr(X_j=-q)=1-\theta$, where $0<\theta<1$. Put $S_0=0$ and $S_j=X_1+\cdots+X_j$. Every returning walk of length $hn$ has $qn$ upward and $pn$ downward steps. By independence, every such walk has the same probability $\theta^{qn}(1-\theta)^{pn}$. The order of its steps does not play a role. After multiplication by the weak and strict excursion counts, we obtain
\begin{align}
 \Pr\{S_{hn}=0,\ S_j\ge0\ (0\le j\le hn)\}
 &=E_n^{(p,q)}\theta^{qn}(1-\theta)^{pn}\quad(n\ge0),\label{eqprobweak}\\
 \Pr\{S_{hn}=0,\ S_j>0\ (1\le j<hn)\}
 &=T_n^{(p,q)}\theta^{qn}(1-\theta)^{pn}\quad(n\ge1).\label{eqprobstrict}
\end{align}
Choosing the $qn$ positions of the upward steps determines a returning walk, so there are $\binom{hn}{qn}$ such walks and
\[
 \Pr\{S_{hn}=0\}=\binom{hn}{qn}\theta^{qn}(1-\theta)^{pn}.
\]
Conditional on this return, all these walks are equally likely. The probabilities of being a weak or strict excursion are therefore $E_n^{(p,q)}/\binom{hn}{qn}$ for $n\ge0$ and $T_n^{(p,q)}/\binom{hn}{qn}$ for $n\ge1$, respectively. The probability $\theta^{qn}(1-\theta)^{pn}$ of each individual returning walk cancels in these ratios. Thus the conditional probabilities do not depend on $\theta$, whereas the return probability and the unconditioned excursion probabilities do for $n\ge1$.

Starting at height $p$, define the height process by $Y_0=p$ and $Y_j=p+S_j$, and put
\[
 \tau=\inf\{j\ge1:Y_j\le0\}.
\]
Thus $\tau$ is the first time the height becomes zero or negative. Every strict excursion starts with a $+p$ step. Removing it leaves a path starting at $p$ that stays positive until its final landing at zero; prepending the step reverses this correspondence. A strict excursion of size $n$ therefore gives a path with $hn-1$ steps, of which $qn-1$ are upward and $pn$ are downward. Writing $\Pr_p$ for probability with starting height $p$, we obtain
\begin{equation}\label{eqruin}
 \Pr_p\{\tau=hn-1,\ Y_\tau=0\}
 =T_n^{(p,q)}\theta^{qn-1}(1-\theta)^{pn}\quad(n\ge1).
\end{equation}
We need the strict excursion count here because a weak excursion may return to zero before its final step. When $q>1$, a walk can jump from a positive height to a negative height without first landing on zero. For example, with $p=4$ and $q=3$, the path $4\to1\to-2$ has $\tau=2$ and $Y_\tau=-2$. Formula~\eqref{eqruin} gives the probability that a walk starting at $p$ stays strictly positive until it lands at zero after $hn-1$ steps. It excludes walks that first enter the nonpositive region at a negative height.

Every first entry at or below zero ends with a $-q$ step from one of the heights $1,\ldots,q$. The terminal height is therefore $-s$ for some $s\in\{0,\ldots,q-1\}$. For $K\ge1$, let $N_{K,s}$ be the number of walks starting at $p$ that stay strictly positive before step $K$ and end at $-s$.

Deleting the final $-q$ step of such a walk leaves a path of length $K-1$ from $p$ to $q-s$. Prepending a $+p$ step from zero gives a positive meander of length $K$ ending at $q-s$. Conversely, every nonempty positive meander begins with $+p$; removing that step and appending $-q$ reverses the construction. Hence
\begin{equation}\label{eqfirstentrycounts}
 \begin{aligned}
 N_{K,s}&=[t^K x^{q-s}]M_{>}(t,x)\\
        &=[t^K x^{q-s}]\exp L_+(t,x).
 \end{aligned}
\end{equation}
Expanding the exponential as in~\eqref{eqexponentialpartitions} expresses this coefficient as a finite sum of products indexed by partitions of $q-s$.

If a counted path has $u$ upward and $d$ downward steps, then $u+d=K$ and $p+pu-qd=-s$. Solving gives the step counts
\[
 u_{K,s}=\frac{qK-p-s}{h},\qquad d_{K,s}=K-u_{K,s}.
\]
When $u_{K,s}$ is an integer between $0$ and $K$, every such path has probability $\theta^{u_{K,s}}(1-\theta)^{K-u_{K,s}}$. Therefore
\begin{equation}\label{eqfirstentryprob}
 \Pr_p\{\tau=K,\ Y_\tau=-s\}
 =N_{K,s}\theta^{u_{K,s}}(1-\theta)^{K-u_{K,s}}.
\end{equation}
If $u_{K,s}$ is not an integer in this range, the event has probability zero. Summing the disjoint terminal-height events gives
\[
 \Pr_p\{\tau=K\}
 =\sum_{s=0}^{q-1}\Pr_p\{\tau=K,\ Y_\tau=-s\}.
\]
For $s=0$ and $K=hn-1$, the coefficient in~\eqref{eqfirstentrycounts} is $T_n^{(p,q)}$ by~\eqref{eqendstrict}, so~\eqref{eqfirstentryprob} includes~\eqref{eqruin}.

\par\noindent\begin{minipage}{\linewidth}
For jumps $+4,-3$, starting at height $4$, the possible last steps are $3\to0$, $2\to-1$ and $1\to-2$. The corresponding earliest entry times are $6$, $4$ and $2$. There are respectively five, two and one paths at these times, giving
\begin{center}
\begin{tabular}{@{}ccc@{}}
\toprule
$K$ & $Y_\tau$ & $\Pr_4\{\tau=K,\ Y_\tau=-s\}$\\
\midrule
$2$ & $-2$ & $(1-\theta)^2$\\
$4$ & $-1$ & $2\theta(1-\theta)^3$\\
$6$ & $0$ & $5\theta^2(1-\theta)^4$\\
\bottomrule
\end{tabular}
\par\smallskip
{\small First-entry probabilities for jumps $+4,-3$ at the earliest possible time for each terminal height.\par}
\end{center}
\end{minipage}\par
For this application, we drew inspiration from studies of decision-making under additive wealth dynamics~\cite{VVMG} and from related work with a ruin boundary~\cite{VG}.

\section{Concluding remarks}\label{secconclusion}
For fixed jumps $+p$ and $-q$, the excursion counts are weighted sums of products of binomial expressions, equivalently generalized Fuss--Catalan coefficients. Both formulas sum over partitions of $q$, independently of the path size $n$. By exchanging the jump magnitudes, the smaller magnitude may be used. The indices in each product satisfy certain divisibility conditions. The same endpoint factorization gives the alternating weak expansion and the positive strict expansion. Multiplying each selected Fuss--Catalan factor by $q$ in the weak formula or $p$ in the strict formula gives a positive integer. The remaining weight $1/z_\lambda$ can still make an individual partition contribution fractional.

These formulas complement the Grossman--Bizley expansions indexed by partitions of the path size. Section~\ref{secexamples} and Appendix~\ref{appformulas} give explicit binomial identities for the corresponding excursion sequences. A direct combinatorial interpretation of the complete weighted partition contributions, together with refinements that exploit the positivity of the strict formula, remains a subject for further work.

\section{Acknowledgments}\label{secpreparation}
The initial intuition and research direction came from the authors. ChatGPT 6 contributed substantially to developing the theory and completing the proof, and assisted with the exposition, reference searches, and arithmetic verification. Afterwards, it helped to check for inconsistencies in the arguments, and it was used to make the LaTeX formatting easier. The final manuscript was written in the authors' own words, and the authors independently checked the proofs and take responsibility for the results.

\appendix
\section{Expanded binomial formulas}\label{appformulas}
\begingroup
\setlength{\jot}{1pt}
\setlength{\abovedisplayskip}{5pt plus 1pt minus 1pt}
\setlength{\belowdisplayskip}{5pt plus 1pt minus 1pt}
\setlength{\abovedisplayshortskip}{3pt plus 1pt}
\setlength{\belowdisplayshortskip}{3pt plus 1pt minus 1pt}
\allowdisplaybreaks[0]
Corollary~\ref{corshifted} gives the following explicit binomial formulas for all jump pairs $+6,-q$ with $1\le q<6$ and $+7,-q$ with $1\le q<7$. All restrictions on the summation indices are already incorporated into the displayed formulas. Every summation index is a nonnegative integer; if the displayed constraint has no solution, the sum is zero. Within each subsection, $E_n$ and $T_n$ denote the weak and strict excursion counts for the indicated jump pair. The formulas for $E_n$ hold for $n\ge0$, and those for $T_n$ hold for $n\ge1$, with $T_0=0$.

Size $n$ refers to the normalized jumps. In particular, $(6,2)$, $(6,3)$, and $(6,4)$ become $(3,1)$, $(2,1)$, and $(3,2)$, respectively.

\subsection{\texorpdfstring{Steps $+6$ and $-1$}{Steps +6 and -1}}

A path of size $n$ has length $7n$, with $n$ upward and $6n$ downward steps.

\nopagebreak[4]
\noindent\begin{minipage}[c]{0.495\linewidth}
\begin{align}\label{eqexpanded9}
E_n={}&\frac{\binom{7n+1}{n}}{7n+1}.
\end{align}
\end{minipage}\hfill
\begin{minipage}[c]{0.495\linewidth}
\begin{align}\label{eqexpanded10}
T_n={}&\frac{\binom{7n-1}{n}}{7n-1}.
\end{align}
\end{minipage}

\subsection{\texorpdfstring{Steps $+6$ and $-2$}{Steps +6 and -2}}

A path of size $n$ has length $4n$, with $n$ upward and $3n$ downward steps.

After normalization, the jumps are $+3$ and $-1$.

\noindent\begin{minipage}[c]{0.495\linewidth}
\begin{align}\label{eqexpanded11}
E_n={}&\frac{\binom{4n+1}{n}}{4n+1}.
\end{align}
\end{minipage}\hfill
\begin{minipage}[c]{0.495\linewidth}
\begin{align}\label{eqexpanded12}
T_n={}&\frac{\binom{4n-1}{n}}{4n-1}.
\end{align}
\end{minipage}

\subsection{\texorpdfstring{Steps $+6$ and $-3$}{Steps +6 and -3}}

A path of size $n$ has length $3n$, with $n$ upward and $2n$ downward steps.

After normalization, the jumps are $+2$ and $-1$.

\noindent\begin{minipage}[c]{0.495\linewidth}
\begin{align}\label{eqexpanded13}
E_n={}&\frac{\binom{3n+1}{n}}{3n+1}.
\end{align}
\end{minipage}\hfill
\begin{minipage}[c]{0.495\linewidth}
\begin{align}\label{eqexpanded14}
T_n={}&\frac{\binom{3n-1}{n}}{3n-1}.
\end{align}
\end{minipage}

\subsection{\texorpdfstring{Steps $+6$ and $-4$}{Steps +6 and -4}}

A path of size $n$ has length $5n$, with $2n$ upward and $3n$ downward steps.

After normalization, the jumps are $+3$ and $-2$.

\noindent\begin{minipage}[c]{0.495\linewidth}
\begin{align}\label{eqexpanded15}
E_n={}&\frac{\binom{5n+1}{2n}}{5n+1}-\frac{1}{2}\sum_{i+j=n-1}\frac{\binom{5i+3}{2i+1}}{5i+3}\frac{\binom{5j+3}{2j+1}}{5j+3}.
\end{align}
\end{minipage}\hfill
\begin{minipage}[c]{0.495\linewidth}
\begin{align}\label{eqexpanded16}
T_n={}&\frac{\binom{5n-1}{2n}}{5n-1}+\frac{1}{2}\sum_{i+j=n-1}\frac{\binom{5i+2}{2i+1}}{5i+2}\frac{\binom{5j+2}{2j+1}}{5j+2}.
\end{align}
\end{minipage}

\newpage
\subsection{\texorpdfstring{Steps $+6$ and $-5$}{Steps +6 and -5}}

A path of size $n$ has length $11n$, with $5n$ upward and $6n$ downward steps.

\begin{align}\label{eqexpanded17}
E_n={}&\frac{\binom{11n+1}{5n}}{11n+1}-\sum_{i+j=n-1}\frac{\binom{11i+3}{5i+1}}{11i+3}\frac{\binom{11j+9}{5j+4}}{11j+9}-\sum_{i+j=n-1}\frac{\binom{11i+5}{5i+2}}{11i+5}\frac{\binom{11j+7}{5j+3}}{11j+7}\notag\\
&+\frac{1}{2}\sum_{i+j+k=n-2}\frac{\binom{11i+5}{5i+2}}{11i+5}\frac{\binom{11j+9}{5j+4}}{11j+9}\frac{\binom{11k+9}{5k+4}}{11k+9}+\frac{1}{2}\sum_{i+j+k=n-2}\frac{\binom{11i+7}{5i+3}}{11i+7}\frac{\binom{11j+7}{5j+3}}{11j+7}\frac{\binom{11k+9}{5k+4}}{11k+9}\notag\\
&-\frac{1}{6}\sum_{i+j+k+\ell=n-3}\frac{\binom{11i+7}{5i+3}}{11i+7}\frac{\binom{11j+9}{5j+4}}{11j+9}\frac{\binom{11k+9}{5k+4}}{11k+9}\frac{\binom{11\ell+9}{5\ell+4}}{11\ell+9}\notag\\
&+\frac{1}{120}\sum_{i+j+k+\ell+m=n-4}\frac{\binom{11i+9}{5i+4}}{11i+9}\frac{\binom{11j+9}{5j+4}}{11j+9}\frac{\binom{11k+9}{5k+4}}{11k+9}\frac{\binom{11\ell+9}{5\ell+4}}{11\ell+9}\frac{\binom{11m+9}{5m+4}}{11m+9}.
\end{align}

\begin{align}\label{eqexpanded18}
T_n={}&\frac{\binom{11n-1}{5n}}{11n-1}+\sum_{i+j=n-1}\frac{\binom{11i+8}{5i+4}}{11i+8}\frac{\binom{11j+2}{5j+1}}{11j+2}+\sum_{i+j=n-1}\frac{\binom{11i+6}{5i+3}}{11i+6}\frac{\binom{11j+4}{5j+2}}{11j+4}\notag\\
&+\frac{1}{2}\sum_{i+j+k=n-1}\frac{\binom{11i+6}{5i+3}}{11i+6}\frac{\binom{11j+2}{5j+1}}{11j+2}\frac{\binom{11k+2}{5k+1}}{11k+2}+\frac{1}{2}\sum_{i+j+k=n-1}\frac{\binom{11i+4}{5i+2}}{11i+4}\frac{\binom{11j+4}{5j+2}}{11j+4}\frac{\binom{11k+2}{5k+1}}{11k+2}\notag\\
&+\frac{1}{6}\sum_{i+j+k+\ell=n-1}\frac{\binom{11i+4}{5i+2}}{11i+4}\frac{\binom{11j+2}{5j+1}}{11j+2}\frac{\binom{11k+2}{5k+1}}{11k+2}\frac{\binom{11\ell+2}{5\ell+1}}{11\ell+2}\notag\\
&+\frac{1}{120}\sum_{i+j+k+\ell+m=n-1}\frac{\binom{11i+2}{5i+1}}{11i+2}\frac{\binom{11j+2}{5j+1}}{11j+2}\frac{\binom{11k+2}{5k+1}}{11k+2}\frac{\binom{11\ell+2}{5\ell+1}}{11\ell+2}\frac{\binom{11m+2}{5m+1}}{11m+2}.
\end{align}

\subsection{\texorpdfstring{Steps $+7$ and $-1$}{Steps +7 and -1}}

A path of size $n$ has length $8n$, with $n$ upward and $7n$ downward steps.

\noindent\begin{minipage}[c]{0.495\linewidth}
\begin{align}\label{eqexpanded19}
E_n={}&\frac{\binom{8n+1}{n}}{8n+1}.
\end{align}
\end{minipage}\hfill
\begin{minipage}[c]{0.495\linewidth}
\begin{align}\label{eqexpanded20}
T_n={}&\frac{\binom{8n-1}{n}}{8n-1}.
\end{align}
\end{minipage}

\subsection{\texorpdfstring{Steps $+7$ and $-2$}{Steps +7 and -2}}

A path of size $n$ has length $9n$, with $2n$ upward and $7n$ downward steps.

\noindent\begin{minipage}[c]{0.495\linewidth}
\begin{align}\label{eqexpanded21}
E_n={}&\frac{\binom{9n+1}{2n}}{9n+1}-\frac{1}{2}\sum_{i+j=n-1}\frac{\binom{9i+5}{2i+1}}{9i+5}\frac{\binom{9j+5}{2j+1}}{9j+5}.
\end{align}
\end{minipage}\hfill
\begin{minipage}[c]{0.495\linewidth}
\begin{align}\label{eqexpanded22}
T_n={}&\frac{\binom{9n-1}{2n}}{9n-1}+\frac{1}{2}\sum_{i+j=n-1}\frac{\binom{9i+4}{2i+1}}{9i+4}\frac{\binom{9j+4}{2j+1}}{9j+4}.
\end{align}
\end{minipage}

\newpage
\subsection{\texorpdfstring{Steps $+7$ and $-3$}{Steps +7 and -3}}

A path of size $n$ has length $10n$, with $3n$ upward and $7n$ downward steps.

\begin{align}\label{eqexpanded23}
E_n={}&\frac{\binom{10n+1}{3n}}{10n+1}-\sum_{i+j=n-1}\frac{\binom{10i+4}{3i+1}}{10i+4}\frac{\binom{10j+7}{3j+2}}{10j+7}+\frac{1}{6}\sum_{i+j+k=n-2}\frac{\binom{10i+7}{3i+2}}{10i+7}\frac{\binom{10j+7}{3j+2}}{10j+7}\frac{\binom{10k+7}{3k+2}}{10k+7}.
\end{align}

\begin{align}\label{eqexpanded24}
T_n={}&\frac{\binom{10n-1}{3n}}{10n-1}+\sum_{i+j=n-1}\frac{\binom{10i+6}{3i+2}}{10i+6}\frac{\binom{10j+3}{3j+1}}{10j+3}+\frac{1}{6}\sum_{i+j+k=n-1}\frac{\binom{10i+3}{3i+1}}{10i+3}\frac{\binom{10j+3}{3j+1}}{10j+3}\frac{\binom{10k+3}{3k+1}}{10k+3}.
\end{align}

\subsection{\texorpdfstring{Steps $+7$ and $-4$}{Steps +7 and -4}}

A path of size $n$ has length $11n$, with $4n$ upward and $7n$ downward steps.

\begin{align}\label{eqexpanded25}
E_n={}&\frac{\binom{11n+1}{4n}}{11n+1}-\sum_{i+j=n-1}\frac{\binom{11i+9}{4i+3}}{11i+9}\frac{\binom{11j+3}{4j+1}}{11j+3}-\frac{1}{2}\sum_{i+j=n-1}\frac{\binom{11i+6}{4i+2}}{11i+6}\frac{\binom{11j+6}{4j+2}}{11j+6}\notag\\
&+\frac{1}{2}\sum_{i+j+k=n-1}\frac{\binom{11i+6}{4i+2}}{11i+6}\frac{\binom{11j+3}{4j+1}}{11j+3}\frac{\binom{11k+3}{4k+1}}{11k+3}\notag\\
&-\frac{1}{24}\sum_{i+j+k+\ell=n-1}\frac{\binom{11i+3}{4i+1}}{11i+3}\frac{\binom{11j+3}{4j+1}}{11j+3}\frac{\binom{11k+3}{4k+1}}{11k+3}\frac{\binom{11\ell+3}{4\ell+1}}{11\ell+3}.
\end{align}

\begin{align}\label{eqexpanded26}
T_n={}&\frac{\binom{11n-1}{4n}}{11n-1}+\sum_{i+j=n-1}\frac{\binom{11i+2}{4i+1}}{11i+2}\frac{\binom{11j+8}{4j+3}}{11j+8}+\frac{1}{2}\sum_{i+j=n-1}\frac{\binom{11i+5}{4i+2}}{11i+5}\frac{\binom{11j+5}{4j+2}}{11j+5}\notag\\
&+\frac{1}{2}\sum_{i+j+k=n-2}\frac{\binom{11i+5}{4i+2}}{11i+5}\frac{\binom{11j+8}{4j+3}}{11j+8}\frac{\binom{11k+8}{4k+3}}{11k+8}\notag\\
&+\frac{1}{24}\sum_{i+j+k+\ell=n-3}\frac{\binom{11i+8}{4i+3}}{11i+8}\frac{\binom{11j+8}{4j+3}}{11j+8}\frac{\binom{11k+8}{4k+3}}{11k+8}\frac{\binom{11\ell+8}{4\ell+3}}{11\ell+8}.
\end{align}

\newpage
\subsection{\texorpdfstring{Steps $+7$ and $-5$}{Steps +7 and -5}}

A path of size $n$ has length $12n$, with $5n$ upward and $7n$ downward steps.

\begin{align}\label{eqexpanded27}
E_n={}&\frac{\binom{12n+1}{5n}}{12n+1}-\sum_{i+j=n-1}\frac{\binom{12i+8}{5i+3}}{12i+8}\frac{\binom{12j+5}{5j+2}}{12j+5}-\sum_{i+j=n-1}\frac{\binom{12i+3}{5i+1}}{12i+3}\frac{\binom{12j+10}{5j+4}}{12j+10}\notag\\
&+\frac{1}{2}\sum_{i+j+k=n-1}\frac{\binom{12i+3}{5i+1}}{12i+3}\frac{\binom{12j+5}{5j+2}}{12j+5}\frac{\binom{12k+5}{5k+2}}{12k+5}+\frac{1}{2}\sum_{i+j+k=n-2}\frac{\binom{12i+10}{5i+4}}{12i+10}\frac{\binom{12j+10}{5j+4}}{12j+10}\frac{\binom{12k+5}{5k+2}}{12k+5}\notag\\
&-\frac{1}{6}\sum_{i+j+k+\ell=n-2}\frac{\binom{12i+10}{5i+4}}{12i+10}\frac{\binom{12j+5}{5j+2}}{12j+5}\frac{\binom{12k+5}{5k+2}}{12k+5}\frac{\binom{12\ell+5}{5\ell+2}}{12\ell+5}\notag\\
&+\frac{1}{120}\sum_{i+j+k+\ell+m=n-2}\frac{\binom{12i+5}{5i+2}}{12i+5}\frac{\binom{12j+5}{5j+2}}{12j+5}\frac{\binom{12k+5}{5k+2}}{12k+5}\frac{\binom{12\ell+5}{5\ell+2}}{12\ell+5}\frac{\binom{12m+5}{5m+2}}{12m+5}.
\end{align}

\begin{align}\label{eqexpanded28}
T_n={}&\frac{\binom{12n-1}{5n}}{12n-1}+\sum_{i+j=n-1}\frac{\binom{12i+4}{5i+2}}{12i+4}\frac{\binom{12j+7}{5j+3}}{12j+7}+\sum_{i+j=n-1}\frac{\binom{12i+9}{5i+4}}{12i+9}\frac{\binom{12j+2}{5j+1}}{12j+2}\notag\\
&+\frac{1}{2}\sum_{i+j+k=n-2}\frac{\binom{12i+9}{5i+4}}{12i+9}\frac{\binom{12j+7}{5j+3}}{12j+7}\frac{\binom{12k+7}{5k+3}}{12k+7}+\frac{1}{2}\sum_{i+j+k=n-1}\frac{\binom{12i+2}{5i+1}}{12i+2}\frac{\binom{12j+2}{5j+1}}{12j+2}\frac{\binom{12k+7}{5k+3}}{12k+7}\notag\\
&+\frac{1}{6}\sum_{i+j+k+\ell=n-2}\frac{\binom{12i+2}{5i+1}}{12i+2}\frac{\binom{12j+7}{5j+3}}{12j+7}\frac{\binom{12k+7}{5k+3}}{12k+7}\frac{\binom{12\ell+7}{5\ell+3}}{12\ell+7}\notag\\
&+\frac{1}{120}\sum_{i+j+k+\ell+m=n-3}\frac{\binom{12i+7}{5i+3}}{12i+7}\frac{\binom{12j+7}{5j+3}}{12j+7}\frac{\binom{12k+7}{5k+3}}{12k+7}\frac{\binom{12\ell+7}{5\ell+3}}{12\ell+7}\frac{\binom{12m+7}{5m+3}}{12m+7}.
\end{align}

\newpage
\subsection{\texorpdfstring{Steps $+7$ and $-6$}{Steps +7 and -6}}

A path of size $n$ has length $13n$, with $6n$ upward and $7n$ downward steps.

\begin{align}\label{eqexpanded29}
E_n={}&\frac{\binom{13n+1}{6n}}{13n+1}-\sum_{i+j=n-1}\frac{\binom{13i+3}{6i+1}}{13i+3}\frac{\binom{13j+11}{6j+5}}{13j+11}-\sum_{i+j=n-1}\frac{\binom{13i+5}{6i+2}}{13i+5}\frac{\binom{13j+9}{6j+4}}{13j+9}\notag\\
&-\frac{1}{2}\sum_{i+j=n-1}\frac{\binom{13i+7}{6i+3}}{13i+7}\frac{\binom{13j+7}{6j+3}}{13j+7}+\frac{1}{2}\sum_{i+j+k=n-2}\frac{\binom{13i+5}{6i+2}}{13i+5}\frac{\binom{13j+11}{6j+5}}{13j+11}\frac{\binom{13k+11}{6k+5}}{13k+11}\notag\\
&+\sum_{i+j+k=n-2}\frac{\binom{13i+7}{6i+3}}{13i+7}\frac{\binom{13j+9}{6j+4}}{13j+9}\frac{\binom{13k+11}{6k+5}}{13k+11}+\frac{1}{6}\sum_{i+j+k=n-2}\frac{\binom{13i+9}{6i+4}}{13i+9}\frac{\binom{13j+9}{6j+4}}{13j+9}\frac{\binom{13k+9}{6k+4}}{13k+9}\notag\\
&-\frac{1}{6}\sum_{i+j+k+\ell=n-3}\frac{\binom{13i+7}{6i+3}}{13i+7}\frac{\binom{13j+11}{6j+5}}{13j+11}\frac{\binom{13k+11}{6k+5}}{13k+11}\frac{\binom{13\ell+11}{6\ell+5}}{13\ell+11}\notag\\
&-\frac{1}{4}\sum_{i+j+k+\ell=n-3}\frac{\binom{13i+9}{6i+4}}{13i+9}\frac{\binom{13j+9}{6j+4}}{13j+9}\frac{\binom{13k+11}{6k+5}}{13k+11}\frac{\binom{13\ell+11}{6\ell+5}}{13\ell+11}\notag\\
&+\frac{1}{24}\sum_{i+j+k+\ell+m=n-4}\frac{\binom{13i+9}{6i+4}}{13i+9}\frac{\binom{13j+11}{6j+5}}{13j+11}\frac{\binom{13k+11}{6k+5}}{13k+11}\frac{\binom{13\ell+11}{6\ell+5}}{13\ell+11}\frac{\binom{13m+11}{6m+5}}{13m+11}\notag\\
&-\frac{1}{720}\sum_{i+j+k+\ell+m+s=n-5}\frac{\binom{13i+11}{6i+5}}{13i+11}\frac{\binom{13j+11}{6j+5}}{13j+11}\frac{\binom{13k+11}{6k+5}}{13k+11}\frac{\binom{13\ell+11}{6\ell+5}}{13\ell+11}\frac{\binom{13m+11}{6m+5}}{13m+11}\frac{\binom{13s+11}{6s+5}}{13s+11}.
\end{align}

\begin{align}\label{eqexpanded30}
T_n={}&\frac{\binom{13n-1}{6n}}{13n-1}+\sum_{i+j=n-1}\frac{\binom{13i+10}{6i+5}}{13i+10}\frac{\binom{13j+2}{6j+1}}{13j+2}+\sum_{i+j=n-1}\frac{\binom{13i+8}{6i+4}}{13i+8}\frac{\binom{13j+4}{6j+2}}{13j+4}\notag\\
&+\frac{1}{2}\sum_{i+j=n-1}\frac{\binom{13i+6}{6i+3}}{13i+6}\frac{\binom{13j+6}{6j+3}}{13j+6}+\frac{1}{2}\sum_{i+j+k=n-1}\frac{\binom{13i+8}{6i+4}}{13i+8}\frac{\binom{13j+2}{6j+1}}{13j+2}\frac{\binom{13k+2}{6k+1}}{13k+2}\notag\\
&+\sum_{i+j+k=n-1}\frac{\binom{13i+6}{6i+3}}{13i+6}\frac{\binom{13j+4}{6j+2}}{13j+4}\frac{\binom{13k+2}{6k+1}}{13k+2}+\frac{1}{6}\sum_{i+j+k=n-1}\frac{\binom{13i+4}{6i+2}}{13i+4}\frac{\binom{13j+4}{6j+2}}{13j+4}\frac{\binom{13k+4}{6k+2}}{13k+4}\notag\\
&+\frac{1}{6}\sum_{i+j+k+\ell=n-1}\frac{\binom{13i+6}{6i+3}}{13i+6}\frac{\binom{13j+2}{6j+1}}{13j+2}\frac{\binom{13k+2}{6k+1}}{13k+2}\frac{\binom{13\ell+2}{6\ell+1}}{13\ell+2}\notag\\
&+\frac{1}{4}\sum_{i+j+k+\ell=n-1}\frac{\binom{13i+4}{6i+2}}{13i+4}\frac{\binom{13j+4}{6j+2}}{13j+4}\frac{\binom{13k+2}{6k+1}}{13k+2}\frac{\binom{13\ell+2}{6\ell+1}}{13\ell+2}\notag\\
&+\frac{1}{24}\sum_{i+j+k+\ell+m=n-1}\frac{\binom{13i+4}{6i+2}}{13i+4}\frac{\binom{13j+2}{6j+1}}{13j+2}\frac{\binom{13k+2}{6k+1}}{13k+2}\frac{\binom{13\ell+2}{6\ell+1}}{13\ell+2}\frac{\binom{13m+2}{6m+1}}{13m+2}\notag\\
&+\frac{1}{720}\sum_{i+j+k+\ell+m+s=n-1}\frac{\binom{13i+2}{6i+1}}{13i+2}\frac{\binom{13j+2}{6j+1}}{13j+2}\frac{\binom{13k+2}{6k+1}}{13k+2}\frac{\binom{13\ell+2}{6\ell+1}}{13\ell+2}\frac{\binom{13m+2}{6m+1}}{13m+2}\frac{\binom{13s+2}{6s+1}}{13s+2}.
\end{align}

\endgroup

2020 \emph{Mathematics Subject Classification}: Primary 05A15; Secondary 05A10, 05A19, 60G50.

\emph{Keywords}: generalized Fuss--Catalan number, rational Dyck path, integer sequence, strict excursion, convolution identity, integer partition.

(Concerned with sequences
\seqnum{A001764},
\seqnum{A002293},
\seqnum{A002296},
\seqnum{A006013},
\seqnum{A006632},
\seqnum{A007556},
\seqnum{A060941},
\seqnum{A130565},
\seqnum{A234466},
\seqnum{A293946},
\seqnum{A300386},
\seqnum{A300387},
\seqnum{A300390}, and
\seqnum{A300391}.)

\end{document}